\documentclass[10pt,oneside,reqno]{article}
\usepackage{authblk}
\usepackage[utf8]{inputenc}
\usepackage[total={5in, 7.7in}]{geometry}
\usepackage{amsmath}
\allowdisplaybreaks
\usepackage{amsthm,ifthen, amsfonts, amssymb,
srcltx, amsopn, color}
\let\oldproofname=\proofname
\renewcommand{\proofname}{\rm\bf{\oldproofname}}
\usepackage{mathabx}
\usepackage{amssymb}
\usepackage[mathscr]{euscript}
\usepackage{enumitem}

\setlist{noitemsep}
\usepackage{graphicx}
\usepackage{fullpage}
\usepackage{tikz}
\usepackage{color}
\usepackage{multicol}
\usepackage{mathrsfs}
\usepackage{hyperref}
\usepackage[normalem]{ulem}
\usepackage{float}
\usepackage[font=footnotesize,labelsep=period,justification=justified]{caption}
\usepackage{mathabx}
\usepackage{hieroglf}
\usepackage[T1]{fontenc}
\usepackage[cmtip,arrow]{xy}
\usepackage{pb-diagram, pb-xy}
\usepackage{overpic}
\usepackage{verbatim}
\usepackage{calc}
\usepackage{transparent}

\usepackage{titlesec}

\titleformat*{\section}{\normalsize\bfseries}
\titleformat*{\subsection}{\normalsize\bfseries}
\titleformat*{\subsubsection}{\normalsize}
\titlelabel{\thetitle.\hspace{0.5em}}

\definecolor{Green}{RGB}{30, 150, 30}

\usetikzlibrary{arrows}
\usepackage{tikz-cd} 
\tikzcdset{row sep/normal=1cm}
\tikzcdset{column sep/normal=1cm}

\newtheorem{thm}{Theorem}[section]
\newtheorem{prop}[thm]{Proposition}
\newtheorem{lem}[thm]{Lemma}
\newtheorem{cor}[thm]{Corollary}

\theoremstyle{definition}
\newtheorem{defn}[thm]{Definition}
\theoremstyle{definition}
\newtheorem{ex}[thm]{Example}
\theoremstyle{definition}

\theoremstyle{definition}
\newtheorem{remark}[thm]{Remark}
\theoremstyle{definition}
\newtheorem{warning}[thm]{Warning}
\theoremstyle{definition}
\newtheorem{question}[thm]{Question}
\theoremstyle{definition}
\newtheorem{notation}[thm]{Notation}
\theoremstyle{definition}
\newtheorem{convention}[thm]{Convention}
\theoremstyle{definition}
\newtheorem{construction}[thm]{Construction}
\theoremstyle{definition}
\newtheorem{algorithm}[thm]{Algorithm}
\theoremstyle{definition}

\theoremstyle{plain}

\newtheorem{thmi}{Theorem}

\newtheorem{questioni}[thmi]{Question}
\newtheorem{propi}[thmi]{Proposition}

\newcommand*\Tot{\operatorname{Tot}}

\newcommand*\mf[1]{\mathfrak{#1}}
\newcommand*\mc[1]{\mathcal{#1}}

\newcommand{\dist}{\textup{\textsf{d}}}

\newcommand{\tsh}[1]{\left\{\kern-.7ex\left\{#1\right\}\kern-.7ex\right\}}

\newcommand*\link{\operatorname{link}}

\newcommand*\sub[1]{\langle {#1} \rangle}
\newcommand{\llangle}{\left\langle\!\left\langle}
\newcommand{\rrangle}{\right\rangle\!\right\rangle}
\newcommand*\nsub[1]{\llangle {#1} \rrangle}
\title{\normalsize \bf GRAPH OF GROUPS DECOMPOSITIONS OF GRAPH BRAID GROUPS}
\author{\footnotesize DANIEL BERLYNE}
\affil{\emph{School of Mathematics, University of Bristol, Woodland Road} \\
\emph{Bristol, BS8 1UG, United Kingdom} \\
\emph{daniel.berlyne@bristol.ac.uk}}
\date{}

\begin{document}

\maketitle

\noindent We provide an explicit construction that allows one to easily decompose a graph braid group as a graph of groups. This allows us to compute the braid groups of a wide range of graphs, as well as providing two general criteria for a graph braid group to split as a non-trivial free product, answering two questions of Genevois. We also use this to distinguish certain right-angled Artin groups and graph braid groups. Additionally, we provide an explicit example of a graph braid group that is relatively hyperbolic, but is not hyperbolic relative to braid groups of proper subgraphs. This answers another question of Genevois in the negative.\\



\section{Introduction}

Given a topological space $X$, one can construct the \emph{configuration space} $C_{n}^{\text{top}}(X)$ of $n$ particles on $X$ by taking the direct product of $n$ copies of $X$ and removing the diagonal consisting of all tuples where two coordinates coincide. Informally, this space tracks the movement of the particles through $X$; removing the diagonal ensures the particles do not collide. One then obtains the \emph{unordered configuration space} $UC_{n}^{\text{top}}(X)$ by taking the quotient by the action of the symmetric group by permutation of the coordinates of $C_n^{\text{top}}(X)$. Finally, the \emph{braid group} $B_{n}(X,S)$ is defined to be the fundamental group of $UC_{n}^{\text{top}}(X)$ with base point $S$ (in general we shall assume $X$ to be connected and drop the base point from our notation). The fundamental group of $C_n^{\text{top}}(X)$ is called the \emph{pure braid group} $PB_n(X)$. This description of braid groups is originally due to Fox and Neuwirth \cite{Fox_braids}.

Classically, the space $X$ is taken to be a disc. However, one may also study braid groups of other spaces. 
Here, we study the case where $X$ is a finite graph $\Gamma$. These so-called \emph{graph braid groups} were first developed by Abrams \cite{AbramsThesis}, who showed that $B_n(\Gamma)$ can be expressed as the fundamental group of a non-positively curved cube complex $UC_n(\Gamma)$; this was also shown independently by \'{S}wi\k{a}tkowski \cite[Proposition 2.3.1]{Swiatkowski_GBG_homological_estimates}. Results of Crisp--Wiest, and later Genevois, show that these cube complexes are in fact \emph{special} \cite{CrispWiest, Gen_braids}, in the sense of Haglund and Wise \cite{HaglundWise_Special}. 

As fundamental groups of special cube complexes, graph braid groups $B_n(\Gamma)$ embed in right-angled Artin groups \cite[Theorem 1.1]{HaglundWise_Special}. In fact, Sabalka also shows that all right-angled Artin groups embed in graph braid groups \cite[Theorem 1.1]{Sabalka_RAAGs_in_GBGs}. It is therefore natural to ask when a graph braid group is isomorphic to a right-angled Artin group. This question was first studied by Connolly and Doig in the case where $\Gamma$ is a linear tree \cite{ConnollyDoig}, and later by Kim, Ko, and Park, who show that for $n \geq 5$, $B_n(\Gamma)$ is isomorphic to a right-angled Artin group if and only if $\Gamma$ does not contain a subgraph homeomorphic to the letter ``A'' or a certain tree \cite[Theorems A, B]{KimKoPark_GBG_Artin}. We strengthen this theorem by showing that if $\Gamma$ does not contain a subgraph homeomorphic to the letter ``A'', then $B_n(\Gamma)$ must split as a non-trivial free product. Thus, for $n \geq 5$, $B_n(\Gamma)$ is never isomorphic to a right-angled Artin group with connected defining graph containing at least two vertices.

\begin{thmi}[Theorem \ref{thm:triangle_free_RAAGs_are_not_GBGs}, {\cite[Theorem B]{KimKoPark_GBG_Artin}}]\label{thmi:RAAGs_2}
Let $\Gamma$ and $\Pi$ be finite connected graphs with at least two vertices and let $n \geq 5$. Then $B_n(\Gamma)$ is not isomorphic to the right-angled Artin group $A_\Pi$.
\end{thmi}

In fact, we prove two much more general criteria for a graph braid group to split as a non-trivial free product. In the theorem below, a \emph{flower graph} is a graph obtained by gluing cycles and segments along a single central vertex.

\begin{thmi}[Lemmas \ref{lem:gbg_free_product_1} and \ref{lem:gbg_free_product_2}]\label{thmi:free_prod_criteria}
Let $n \geq 2$ and let $\Gamma$ be a finite graph whose edges are sufficiently subdivided. Suppose one of the following holds:
\begin{itemize}
    \item $\Gamma$ is obtained by gluing a non-segment flower graph $\Phi$ to a connected non-segment graph $\Omega$ along a vertex $v$, where $v$ is either the central vertex of $\Phi$ or a vertex of $\Phi$ of valence one;
    \item $\Gamma$ contains an edge $e$ such that $\Gamma \smallsetminus \mathring{e}$ is connected but $\Gamma \smallsetminus e$ is disconnected, and one of the connected components of $\Gamma \smallsetminus e$ is a segment.
\end{itemize}
Then $B_n(\Gamma) \cong H \ast \mathbb{Z}$ for some non-trivial subgroup $H$ of $B_n(\Gamma)$.
\end{thmi}

It would be interesting to know if the converse is true; the author is not aware of any counterexamples, but a negative answer seems likely. An interesting related question concerns the existence of graph braid groups that split as free products but do not have a splitting with an infinite cyclic free factor. If such graph braid groups did not exist, this would be a surprising result that would for example imply that there is no graph braid group isomorphic to $\mathbb{Z}^2 \ast \mathbb{Z}^2$, however examples appear elusive here, too.

\begin{questioni}
Is the converse of Theorem \ref{thmi:free_prod_criteria} true? Are there any graph braid groups $B_n(\Gamma)$ that split as non-trivial free products but do not have a free splitting of the form $B_n(\Gamma) \cong H \ast \mathbb{Z}$?
\end{questioni}

In light of Theorem \ref{thmi:RAAGs_2}, it is natural to ask in which ways the behavior of graph braid groups is similar to that of right-angled Artin groups and in which ways it differs. One approach to this is to study non-positive curvature properties of graph braid groups. For example, it is known that a right-angled Artin group $A_\Pi$ is hyperbolic if and only if $\Pi$ contains no edges, $A_\Pi$ is relatively hyperbolic if and only if $\Pi$ is disconnected, and $A_\Pi$ is toral relatively hyperbolic if and only if $\Pi$ is disconnected and every connected component is a complete graph. Furthermore, if $A_\Pi$ is relatively hyperbolic, then its peripheral subgroups are right-angled Artin groups whose defining graphs are connected components of $\Pi$. One can obtain similar, albeit more complicated, graphical characterizations of (relative) hyperbolicity in right-angled Coxeter groups $W_\Pi$; see \cite{Moussong,Levcovitz_hypergraph_old}. Once again, the peripheral subgroups appear as right-angled Coxeter groups whose defining graphs are subgraphs of $\Pi$.

Genevois shows that in the case of hyperbolicity, graph braid groups admit a graphical characterization \cite[Theorem 1.1]{Gen_braids}: for connected $\Gamma$, $B_2(\Gamma)$ is hyperbolic if and only if $\Gamma$ does not contain two disjoint cycles; $B_3(\Gamma)$ is hyperbolic if and only if $\Gamma$ is a tree, a sun graph, a flower graph, or a pulsar graph; and for $n \geq 4$, $B_n(\Gamma)$ is hyperbolic if and only if $\Gamma$ is a flower graph. However, one shortcoming of this theorem is that it introduced classes of graphs whose braid groups were unknown: sun graphs and pulsar graphs (braid groups on flower graphs are known to be free by \cite[Corollary 4.7]{Gen_braids}, while braid groups on trees are known to be free for $n=3$ by \cite[Theorems 2.5, 4.3]{FarleySabalka_DMT}). By applying Theorem \ref{thmi:free_prod_criteria}, we are able to (partially) answer a question of Genevois \cite[Question 5.3]{Gen_braids} and thus provide a more complete algebraic characterization of hyperbolicity. The only exception is when $\Gamma$ is a \emph{generalized theta graph}, which proves resistant to computation. Here, a generalized theta graph $\Theta_m$ is a graph obtained by gluing $m$ cycles along a non-trivial segment. 

\begin{thmi}[Theorem \ref{thm:hyperbolicity_and_free_products}, {\cite[Theorem 1.1]{Gen_braids}}]\label{thmi:hyp}
Let $\Gamma$ be a finite connected graph that is not homeomorphic to $\Theta_m$ for any $m \geq 0$. The braid group $B_3(\Gamma)$ is hyperbolic only if $B_3(\Gamma) \cong H \ast \mathbb{Z}$ for some group $H$.
\end{thmi}

Genevois also provides a graphical characterization of toral relative hyperbolicity \cite[Theorem 4.22]{Gen_braids}. Again, this theorem introduced several classes of graphs for which the braid groups were unknown. In the case of $n=4$, this is a finite collection of graphs: the graphs homeomorphic to the letters ``H'', ``A'', and ``$\theta$''. We are able to precisely compute the braid groups of these graphs, completing the algebraic characterization of toral relative hyperbolicity for $n=4$ and answering a question of Genevois \cite[Question 5.6]{Gen_braids}.

\begin{thmi}[Theorem \ref{thm:toral_rel_hyp_and_free_prods}, {\cite[Theorem 4.22]{Gen_braids}}]\label{thmi:rel_hyp}
Let $\Gamma$ be a finite connected graph. The braid group $B_4(\Gamma)$ is toral relatively hyperbolic only if it is either a free group or isomorphic to $F_{10} \ast \mathbb{Z}^2$ or $F_{5} \ast \mathbb{Z}^2$ or an HNN extension of $\mathbb{Z} \ast \mathbb{Z}^2$.
\end{thmi}

It is much more difficult to characterize relative hyperbolicity in general for graph braid groups. Indeed, we show that in some sense it is impossible to obtain a graphical characterization of the form that exists for right-angled Artin and Coxeter groups, by providing an example of a graph braid group that is relatively hyperbolic but is not hyperbolic relative to any braid groups of proper subgraphs (Fig. \ref{figi:counterexample}). This answers a question of Genevois in the negative \cite[follow-up to Question 5.7]{Gen_braids}.

\begin{thmi}[Theorem \ref{thm:counterexample}]\label{thmi:counterexample}
There exists a graph braid group $B_{n}(\Gamma)$ that is hyperbolic relative to a thick, proper subgroup $P$ that is not contained in any graph braid group of the form $B_k(\Lambda,S)$ for $k \leq n$ and $\Lambda \subsetneq \Gamma$. 
\end{thmi}

\begin{figure}[ht]
     \centering
     \def\svgscale{0.5}
\begingroup%
  \makeatletter%
  \providecommand\color[2][]{%
    \errmessage{(Inkscape) Color is used for the text in Inkscape, but the package 'color.sty' is not loaded}%
    \renewcommand\color[2][]{}%
  }%
  \providecommand\transparent[1]{%
    \errmessage{(Inkscape) Transparency is used (non-zero) for the text in Inkscape, but the package 'transparent.sty' is not loaded}%
    \renewcommand\transparent[1]{}%
  }%
  \providecommand\rotatebox[2]{#2}%
  \newcommand*\fsize{\dimexpr\f@size pt\relax}%
  \newcommand*\lineheight[1]{\fontsize{\fsize}{#1\fsize}\selectfont}%
  \ifx\svgwidth\undefined%
    \setlength{\unitlength}{170.38004237bp}%
    \ifx\svgscale\undefined%
      \relax%
    \else%
      \setlength{\unitlength}{\unitlength * \real{\svgscale}}%
    \fi%
  \else%
    \setlength{\unitlength}{\svgwidth}%
  \fi%
  \global\let\svgwidth\undefined%
  \global\let\svgscale\undefined%
  \makeatother%
  \begin{picture}(1,0.9170276)%
    \lineheight{1}%
    \setlength\tabcolsep{0pt}%
    \put(0,0){\includegraphics[width=\unitlength,page=1]{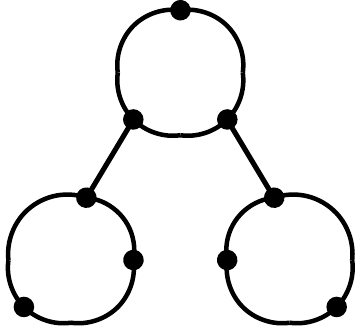}}%
    \put(-0.01332609,0.74141467){\color[rgb]{0,0,0}\makebox(0,0)[lt]{\lineheight{1.25}\smash{\begin{tabular}[t]{l}$\Gamma$\end{tabular}}}}%
  \end{picture}%
\endgroup%
 \hspace{0.5in}
     \def\svgscale{0.75}
\begingroup%
  \makeatletter%
  \providecommand\color[2][]{%
    \errmessage{(Inkscape) Color is used for the text in Inkscape, but the package 'color.sty' is not loaded}%
    \renewcommand\color[2][]{}%
  }%
  \providecommand\transparent[1]{%
    \errmessage{(Inkscape) Transparency is used (non-zero) for the text in Inkscape, but the package 'transparent.sty' is not loaded}%
    \renewcommand\transparent[1]{}%
  }%
  \providecommand\rotatebox[2]{#2}%
  \newcommand*\fsize{\dimexpr\f@size pt\relax}%
  \newcommand*\lineheight[1]{\fontsize{\fsize}{#1\fsize}\selectfont}%
  \ifx\svgwidth\undefined%
    \setlength{\unitlength}{164.09625196bp}%
    \ifx\svgscale\undefined%
      \relax%
    \else%
      \setlength{\unitlength}{\unitlength * \real{\svgscale}}%
    \fi%
  \else%
    \setlength{\unitlength}{\svgwidth}%
  \fi%
  \global\let\svgwidth\undefined%
  \global\let\svgscale\undefined%
  \makeatother%
  \begin{picture}(1,0.98861624)%
    \lineheight{1}%
    \setlength\tabcolsep{0pt}%
    \put(0,0){\includegraphics[width=\unitlength,page=1]{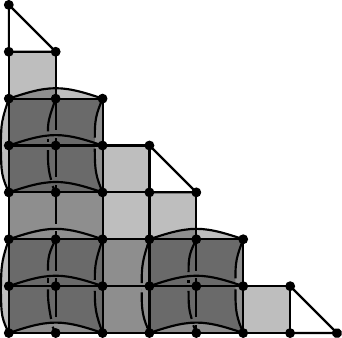}}%
    \put(0.32768203,0.88955853){\color[rgb]{0,0,0}\makebox(0,0)[lt]{\lineheight{1.25}\smash{\begin{tabular}[t]{l}$X = UC_2(\Gamma)$\end{tabular}}}}%
    \put(0,0){\includegraphics[width=\unitlength,page=2]{counterexample.pdf}}%
  \end{picture}%
\endgroup%

     \caption{$B_2(\Gamma)$ is isomorphic to $\pi_1(X)$, which is hyperbolic relative to the subgroup $P$ generated by the fundamental groups of the three shaded tori.}
     \label{figi:counterexample}
 \end{figure}

Note that the peripheral subgroup $P$ in the above example (Fig. \ref{figi:counterexample}) is precisely the group constructed by Croke and Kleiner in \cite[Section 3]{Croke_Kleiner_gp}. In particular, it is isomorphic to the right-angled Artin group $A_\Pi$ where $\Pi$ is a segment of length $3$. This theorem indicates that non-relatively hyperbolic behavior cannot be localized to specific regions of the graph $\Gamma$. Instead, non-relative hyperbolicity is in some sense a property intrinsic to the special cube complex structure.

Theorems \ref{thmi:RAAGs_2}, \ref{thmi:free_prod_criteria}, \ref{thmi:hyp}, \ref{thmi:rel_hyp}, and \ref{thmi:counterexample} are all proved using a technical result on graph of groups decompositions of graph braid groups, which we believe to be of independent interest. Graph of groups decompositions were first considered for pure graph braid groups $PB_n(\Gamma)$ by Abrams \cite{AbramsThesis} and Ghrist \cite{Ghrist_config_spaces_in_robotics}, and more recently Genevois produced a limited result of this flavor for $B_n(\Gamma)$ \cite[Proposition 4.6]{Gen_braids}. In this paper, we use the structure of $UC_n(\Gamma)$ as a special cube complex to produce a general construction that allows one to explicitly compute graph of groups decompositions of graph braid groups. In particular, the vertex groups and edge groups are braid groups on proper subgraphs. By iterating this procedure, one is therefore able to express a graph braid group as a combination of simpler, known graph braid groups.

\begin{thmi}[Theorem \ref{thm:graphs_of_graph_braid_groups}]\label{thmi:thm:graphs_of_graph_braid_groups}
Let $\Gamma$ be a finite, connected, oriented graph whose edges are sufficiently subdivided and let $e_1, \dots, e_m$ be distinct edges of $\Gamma$ sharing a common vertex. The graph braid group $B_n(\Gamma)$ decomposes as a graph of groups $(\mc{G},\Lambda)$, where:
\begin{itemize}
    \item $V(\Lambda)$ is the collection of connected components $K$ of $UC_n(\Gamma \smallsetminus (\mathring{e}_1\cup\dots\cup\mathring{e}_m))$;
    \item $E(\Lambda)$ is the collection of oriented hyperplanes $H$ of $UC_n(\Gamma)$ labeled by some oriented edge $e_i$, where $H \in E(\Lambda)$ has initial and terminal vertices $K \in V(\Lambda)$ and $L \in V(\Lambda)$ if the combinatorial hyperplanes $H^-$ and $H^+$ lie in the components $K$ and $L$ of $UC_n(\Gamma \smallsetminus (\mathring{e}_1\cup\dots\cup\mathring{e}_m))$, respectively;
    \item for each $K \in V(\Lambda)$, we have $G_K = B_n(\Gamma \smallsetminus (\mathring{e}_1\cup\dots\cup\mathring{e}_m), S_K)$ for some $S_K \in K$;
    \item for each $H_i \in E(\Lambda)$ labeled by $e_i$, we have $G_{H_i} = B_{n-1}(\Gamma \smallsetminus e_i, S_{H_i} \cap (\Gamma \smallsetminus e_i))$, for some $S_{H_i}$ in one of the combinatorial hyperplanes $H^\pm_i$;
    \item for each edge $H \in E(\Lambda)$ joining vertices $K,L \in V(\Lambda)$, the monomorphisms $\phi^\pm_{H}$ are induced by the inclusion maps of the combinatorial hyperplanes $H^\pm$ into $K$ and $L$.
\end{itemize}
\end{thmi}

By selecting the edges $e_1, \dots, e_m$ carefully, one may often be able to arrange for the edge groups of $(\mc{G},\Lambda)$ to be trivial. This results in a wide range of graph braid groups that split as non-trivial free products, as shown in Theorem \ref{thmi:free_prod_criteria}. 

This construction also aids in the computation of specific graph braid groups, especially when combined with Propositions \ref{prop:determining_adjacency} and \ref{prop:counting_edges}, in which we give combinatorial criteria for adjacency of two vertices of $\Lambda$, as well as providing a way of counting how many edges connect each pair of vertices of $\Lambda$. Traditionally, graph braid group computations are performed by using Farley and Sabalka's discrete Morse theory to find explicit presentations \cite{FarleySabalka_DMT,FarleySabalka_Presentations}, however in practice these presentations are often highly complex, difficult to compute, and are obfuscated by redundant generators and relators. We believe the graph of groups approach to be both easier to apply and more powerful in many situations, as evidenced by our theorems and the numerous examples we are able to compute in Section \ref{sec:examples}. This is seen most clearly in our computation of $B_n(\Gamma)$ when $\Gamma$ is a \emph{wheel graph} $W_m$, formed by connecting a single central vertex to all vertices of an $m$--cycle.

\begin{propi}[Proposition \ref{prop:wheels}]
Let $W_m$ be the wheel graph with $m$ spokes. Then $B_2(W_m) \cong F_{m+1}$.
\end{propi}

Our proof of this proposition is very short with little background knowledge required. This is in contrast with Farley--Sabalka's proofs \cite[Examples 5.2, 5.4]{FarleySabalka_Presentations}, which only cover the cases $m=2$ and $m=4$ and require much longer computations, producing redundant generators and relators that must be eliminated using Tietze transformations. For example, their computation of $B_2(W_4)$ initially produces a presentation with eleven generators and six relators. 

Despite the significant advantages in the examples described above, the graph of groups technique also has some limitations. The graph of groups decompositions are always very explicit, and by inductively decomposing the braid groups that appear as vertex and edge groups, one can always eventually obtain free groups as vertex and edge groups. However, one often encounters difficulties when trying to understand the monomorphisms. By judiciously choosing edges of the graph $\Gamma$ when applying Theorem \ref{thmi:thm:graphs_of_graph_braid_groups}, one can often avoid this issue by making very simple edge groups, as shown throughout Section \ref{sec:examples}; this is especially true for low numbers of particles. One can also get a better understanding of the monomorphisms if the cube complex can be drawn in a clear way; see Example \ref{ex:HNN}. However, in other cases, it becomes more difficult to analyze the group. As a result, the group $H$ in Theorem \ref{thmi:hyp}, the HNN extension in Theorem \ref{thmi:rel_hyp}, and the braid groups of generalized theta graphs all have explicit descriptions as graphs of groups -- see Proposition \ref{prop:theta_graph} and Question \ref{q:gen_theta} for the latter two -- but it is hard to extract more information about the groups beyond that which is given in the statements of these results. Indeed, the graphs of groups one obtains can sometimes have unexpected properties, such as $B_2(K_5)$ and $B_2(K_{3,3})$, which are surface groups; see Examples \ref{ex:complete} and \ref{ex:complete_bipartite}.

\vspace{3mm}
\noindent\textbf{Outline of the paper.} We begin with some background on the geometry of cube complexes in Section \ref{section:background_cubes}, in particular introducing the notion of a special cube complex. We then introduce graph braid groups in Section \ref{section:background_braids}, showing that they are fundamental groups of special cube complexes and providing some important foundational results. Section \ref{section:background_graphs_of_groups} then introduces the necessary material on graphs of groups.

Our main technical theorem, Theorem \ref{thmi:thm:graphs_of_graph_braid_groups}, is proved in Section \ref{sec:decomposing_gbgs} (Theorem \ref{thm:graphs_of_graph_braid_groups}). We then apply this theorem to compute a number of specific examples of graph braid groups in Section \ref{sec:examples}, including braid groups of radial trees, graphs homeomorphic to the letters ``H'', ``A'', ``$\theta$'', and ``Q'', wheel graphs, complete graphs, and complete bipartite graphs. These examples allow us to prove Theorem \ref{thmi:rel_hyp} (Theorem \ref{thm:toral_rel_hyp_and_free_prods}).

Section \ref{sec:applications} deals with more general applications of Theorem \ref{thmi:thm:graphs_of_graph_braid_groups}. In Section \ref{sec:free_prods} we prove Theorem \ref{thmi:free_prod_criteria}, providing two general criteria for a graph braid group to decompose as a non-trivial free product (Lemmas \ref{lem:gbg_free_product_1} and \ref{lem:gbg_free_product_2}). In Section \ref{sec:rel_hyp}, we prove Theorem \ref{thmi:counterexample} (Theorem \ref{thm:counterexample}), showing that there exist relatively hyperbolic graph braid groups that are not hyperbolic relative to any braid group of a subgraph. We then prove Theorem \ref{thmi:RAAGs_2} (Theorem \ref{thm:triangle_free_RAAGs_are_not_GBGs}) in Section \ref{sec:RAAGs}, showing that braid groups on large numbers of particles cannot be isomorphic to right-angled Artin groups with connected defining graphs.

We conclude by posing some open questions in Section \ref{sec:questions}.

\section{Background}

\subsection{\textit{Cube complexes}}\label{section:background_cubes}

\begin{defn}[Cube complex]
Let $n \geq 0$. An $n$--\emph{cube} is a Euclidean cube $[-1/2,1/2]^{n}$. A \emph{face} of a cube is a subcomplex obtained by restricting one or more of the coordinates to $\pm 1/2$. A \emph{cube complex} is a CW complex where each cell is a cube and the attaching maps are given by isometries along faces. 
\end{defn}

We will often refer to the $0$--cubes of a cube complex $X$ as \emph{vertices}, the $1$--cubes as \emph{edges}, and the $2$--cubes as \emph{squares}. We shall use the piecewise Euclidean metric on $X$ induced by the individual Euclidean cubes, denoted $\dist_X$.

\begin{defn}[Non-positively curved, CAT(0)]\label{defn:non-pos_cube}
Let $X$ be a cube complex. The \emph{link} $\link(v)$ of a vertex $v$ of $X$ is a simplicial complex defined as follows. 
\begin{itemize}
    \item The vertices of $\link(v)$ are the edges of $X$ that are incident at $v$.
    \item $n$ vertices of $\link(v)$ span an $n$--simplex if the corresponding edges of $X$ are faces of a common cube.
\end{itemize}
The complex $\link(v)$ is said to be \emph{flag} if $n$ vertices $v_{1},\dots,v_{n}$ of $\link(v)$ span an $n$--simplex if and only if $v_{i}$ and $v_{j}$ are connected by an edge for all $i \neq j$. A cube complex $X$ is \emph{non-positively curved} if the link of each vertex of $X$ is flag and contains no bigons (that is, no loops consisting of two edges). A cube complex $X$ is \emph{CAT(0)} if it is non-positively curved and simply connected.
\end{defn}

We study cube complexes from the geometric point of view of \emph{hyperplanes}.

\begin{defn}[Mid-cube, hyperplane, combinatorial hyperplane, carrier]
Let $X$ be a cube complex. A \emph{mid-cube} of a cube $C \cong [-1/2,1/2]^{n}$ of $X$ is obtained by restricting one of the coordinates of $C$ to $0$. Construct an equivalence relation $\sim$ on the set of mid-cubes of $X$ by defining $M_1 \sim M_2$ for mid-cubes $M_1, M_2$ if there exists some edge $E$ of $X$ such that $M_i \cap E \neq \emptyset$ for $i = 1,2$, and taking the transitive closure. A \emph{hyperplane} $H$ of $X$ is then defined to be the union of mid-cubes in a $\sim$--equivalence class. 

Each mid-cube of a cube $C \cong [-1/2,1/2]^{n}$ has two isometric associated faces of $C$, obtained by restricting the coordinate to $\pm 1/2$ instead of $0$. Given a hyperplane $H$, construct an equivalence relation $\sim_H$ on the set of faces associated to mid-cubes of $H$ by defining $F_1 \sim_H F_2$ for faces $F_1, F_2$ if $F_1 \cap F_2 \neq \emptyset$ and their associated mid-cubes intersect a common edge of $X$, and taking the transitive closure. A \emph{combinatorial hyperplane} associated to $H$ then defined to be the union of faces in a $\sim_H$--equivalence class. 

The \emph{closed (resp. open) carrier} of a hyperplane $H$ is the union of all closed (resp. open) cubes of $X$ which contain mid-cubes of $H$.
\end{defn}

\begin{convention}
We will almost always be using the closed version of the carrier of a hyperplane $H$, therefore we will refer to this version as simply the `carrier' of $H$.
\end{convention}

A result of Chepoi tells us that carriers and combinatorial hyperplanes form convex subcomplexes of a CAT(0) cube complex \cite[Proof of Proposition 6.6]{Chepoi_graphs_CAT(0)}. In the greater generality of non-positively curved cube complexes, this is not always true, however carriers and combinatorial hyperplanes are still \emph{locally convex}, in the following sense.

\begin{defn}[Local isometry, locally convex]
A map $\phi: Y \rightarrow X$ between non-positively curved cube complexes is a \emph{local isometry} if it is a local injection and for each vertex $y \in Y^{(0)}$ and each pair of vertices $u,v \in Y^{(0)}$ adjacent to $y$, if $\phi(u),\phi(y),\phi(v)$ form a corner of a $2$--cube in $X$, then $u,y,v$ form a corner of a $2$--cube in $Y$. A subcomplex of a non-positively curved cube complex $X$ is \emph{locally convex} if it embeds by a local isometry.
\end{defn}

\begin{prop}\label{prop:local_convexity_of_comb_hyps}
Hyperplane carriers and combinatorial hyperplanes in a non-positively curved cube complex are locally convex.
\end{prop}

\begin{proof}
Let $X$ be a non-positively curved cube complex, let $Y$ be a subcomplex of $X$ that is either a hyperplane carrier or a combinatorial hyperplane, and let $\phi: Y \rightarrow X$ be the natural embedding of $Y$ as a subcomplex of $X$. Note that the universal cover $\tilde{X}$ is a CAT(0) cube complex, therefore hyperplane carriers and combinatorial hyperplanes in $\tilde{X}$ are convex subcomplexes \cite[Proposition 6.6]{Chepoi_graphs_CAT(0)}. 

Let $y \in Y^{(0)}$ and let $u,v \in Y^{(0)}$ be vertices adjacent to $y$ such that $\phi(u),\phi(y),\phi(v)$ form a corner of a $2$--cube $C \subseteq X$. Lift $Y$ to a subcomplex $\tilde{Y}$ of $\tilde{X}$, let $\tilde{u},\tilde{y},\tilde{v}$ be the corresponding lifts of $u,y,v$, let $\psi: \tilde{Y} \rightarrow \tilde{X}$ be the natural embedding of $\tilde{Y}$ as a subcomplex of $\tilde{X}$, and lift $C$ to a $2$--cube $\tilde{C} \subseteq \tilde{X}$ so that $\psi(\tilde{u}),\psi(\tilde{y}),\psi(\tilde{v})$ form a corner of $\tilde{C}$. By convexity of $\tilde{Y}$ in $\tilde{X}$, $\tilde{u},\tilde{y},\tilde{v}$ must also form a corner of a $2$--cube in $\tilde{Y}$. Projecting back down to $Y$, we therefore see that $u,y,v$ form a corner of a $2$--cube in $Y$, hence $Y$ is locally convex in $X$.
\end{proof}

The combinatorial hyperplanes associated to a given hyperplane are also `parallel', in the following sense.

\begin{defn}[Parallelism]
Let $H$ be a hyperplane of a cube complex $X$. We say that $H$ is \emph{dual} to an edge $E$ of $X$ if $H$ contains a mid-cube which intersects $E$. We say that $H$ \emph{crosses} a subcomplex $Y$ of $X$ if there exists some edge $E$ of $Y$ that is dual to $H$. We say that $H$ crosses another hyperplane $H'$ if it crosses a combinatorial hyperplane associated to $H'$. Two subcomplexes $Y,Y'$ of $X$ are \emph{parallel} if each hyperplane $H$ of $X$ crosses $Y$ if and only if it crosses $Y'$.
\end{defn}

Parallelism defines an equivalence relation on the edges of a cube complex $X$. In particular, two edges are in the same equivalence class (or \emph{parallelism class}) if and only if they are dual to the same hyperplane. Therefore, one may instead consider hyperplanes of $X$ to be parallelism classes of edges of $X$. One may also define an orientation $\overrightarrow{H}$ on a hyperplane $H$ by taking an equivalence class of oriented edges. In this case, we say that $\overrightarrow{H}$ is dual to the oriented edges in this class.

\begin{defn}[Osculation]
We say $H$ \emph{directly self-osculates} if there exist two oriented edges $\overrightarrow{E_{1}},\overrightarrow{E_{2}}$ dual to $\overrightarrow{H}$ that have the same initial or terminal vertex but do not span a square. We say two hyperplanes $H_{1},H_{2}$ \emph{inter-osculate} if they intersect and there exist dual edges $E_{1},E_{2}$ of $H_{1},H_{2}$, respectively, such that $E_{1}$ and $E_{2}$ share a common endpoint but do not span a square. See Fig. \ref{fig:osculation}.
\end{defn}

\begin{figure}[ht]
     \centering
     \def\svgscale{1}
     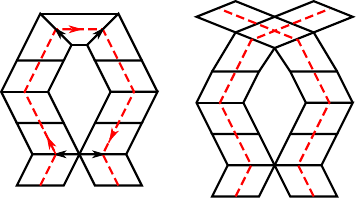
     \caption[Self-osculation and inter-osculation]{A directly self-osculating hyperplane and a pair of inter-osculating hyperplanes.}
     \label{fig:osculation}
 \end{figure}

\begin{defn}[Special]\label{defn:special}
A non-positively curved cube complex $X$ is said to be \emph{special} if its hyperplanes satisfy the following properties.
\begin{enumerate}
    \item Hyperplanes are two-sided; that is, the open carrier of a hyperplane $H$ is homeomorphic to $H \times (-1/2,1/2)$.
    \item Hyperplanes of $X$ do not self-intersect.
    \item Hyperplanes of $X$ do not directly self-osculate.
    \item Hyperplanes of $X$ do not inter-osculate.
\end{enumerate}
\end{defn}

\begin{notation}\label{not:comb_hyp}
Two-sidedness implies combinatorial hyperplanes associated to $H$ are homeomorphic to $H \times \{\pm 1/2\}$. Given an orientation on $H$, we denote the combinatorial hyperplane corresponding to $H \times \{-1/2\}$ by $H^-$ and the combinatorial hyperplane corresponding to $H \times \{1/2\}$ by $H^+$.
\end{notation}

\subsection{\textit{Graph braid groups}}\label{section:background_braids}

Consider a finite collection of particles lying on a finite metric graph $\Gamma$. The \emph{configuration space} of these particles on $\Gamma$ is the collection of all possible ways the particles can be arranged on the graph with no two particles occupying the same location. As we move through the configuration space, the particles move along $\Gamma$, without colliding. If we do not distinguish between each of the different particles, we call this an \emph{unordered} configuration space. A \emph{graph braid group} is the fundamental group of an unordered configuration space. More precisely:

\begin{defn}[Graph braid group]
Let $\Gamma$ be a finite graph, and let $n$ be a positive integer. The \emph{topological configuration space} $C^{\text{top}}_{n}(\Gamma)$ is defined as  \[C^{\text{top}}_{n}(\Gamma) = \Gamma^{n} \smallsetminus D^{\text{top}},\] where $D^{\text{top}} = \{ (x_{1}, \dots, x_{n}) \in \Gamma^{n} \,\,|\,\, x_{i} = x_{j} \text{ for some } i \neq j \}$. The \emph{unordered topological configuration space} $UC^{\text{top}}_{n}(\Gamma)$ is then defined as \[UC^{\text{top}}_{n}(\Gamma) = C^{\text{top}}_{n}(\Gamma) / S_{n},\] where the symmetric group $S_{n}$ acts on $C^{\text{top}}_{n}(\Gamma)$ by permuting its coordinates. We define the \emph{graph braid group} $B_{n}(\Gamma,S)$ as \[B_{n}(\Gamma, S) = \pi_{1}(UC^{\text{top}}_{n}(\Gamma), S),\] where $S \in UC^{\text{top}}_{n}(\Gamma)$ is a fixed base point.
\end{defn}

The base point $S$ in our definition represents an initial configuration of the particles on the graph $\Gamma$. As the particles are unordered, they may always be moved along $\Gamma$ into any other desired initial configuration, so long as the correct number of particles are present in each connected component of $\Gamma$. In particular, if $\Gamma$ is connected, then the graph braid group $B_{n}(\Gamma,S)$ is independent of the choice of base point, and may therefore be denoted simply $B_{n}(\Gamma)$. 

Note that in some sense, the space $UC^{\text{top}}_{n}(\Gamma)$ is almost a cube complex. Indeed, $\Gamma^{n}$ is a cube complex, but removing the diagonal breaks the structure of some of its cubes. By expanding the diagonal slightly, we are able to fix this by ensuring that we are always removing whole cubes.

\begin{defn}[Combinatorial configuration space]
Let $\Gamma$ be a finite graph, and let $n$ be a positive integer. For each $x \in \Gamma$, the \emph{carrier} $c(x)$ of $x$ is the lowest dimensional simplex of $\Gamma$ containing $x$. The \emph{combinatorial configuration space} $C_{n}(\Gamma)$ is defined as \[C_{n}(\Gamma) = \Gamma^{n} \smallsetminus D,\] where $D = \{ (x_{1}, \dots, x_{n}) \in \Gamma^{n} \,\,|\,\, c(x_{i}) \cap c(x_{j}) \neq \emptyset \text{ for some } i \neq j \}$. The \emph{unordered combinatorial configuration space} $UC_{n}(\Gamma)$ is then \[UC_{n}(\Gamma) = C_{n}(\Gamma) / S_{n}.\] The \emph{reduced graph braid group} $RB_{n}(\Gamma,S)$ is defined as \[RB_{n}(\Gamma,S) = \pi_{1}(UC_{n}(\Gamma),S),\] where $S \in UC_{n}(\Gamma)$ is a fixed base point.  
\end{defn}

Removing this new version of the diagonal tells us that two particles cannot occupy the same edge of $\Gamma$. This effectively discretizes the motion of the particles to jumps between vertices, as each particle must fully traverse an edge before another particle may enter.

Observe that $C_{n}(\Gamma)$ is the union of all products $c(x_{1}) \times \dots \times c(x_{n})$ satisfying $c(x_{i}) \cap c(x_{j}) = \emptyset$ for all $i \neq j$. Since the carrier is always a vertex or a closed edge, this defines an $n$--dimensional cube complex, and moreover $C_{n}(\Gamma)$ is compact with finitely many hyperplanes, as $\Gamma$ is a finite graph. It follows that $UC_{n}(\Gamma)$ is also a compact cube complex with finitely many hyperplanes. Indeed, we have the following useful description of the cube complex structure, due to Genevois.

\begin{construction}[{\cite[Section 3]{Gen_braids}}]\label{constr:Gen_UC} The cube complex $UC_n(\Gamma)$ may be constructed as follows.
\begin{itemize}
    \item The vertices of $UC_{n}(\Gamma)$ are the subsets $S$ of $V(\Gamma)$ with size $|S| = n$.
    \item Two vertices $S_{1}$ and $S_{2}$ of $UC_{n}(\Gamma)$ are connected by an edge if their symmetric difference $S_{1} \triangle S_{2}$ is a pair of adjacent vertices of $\Gamma$. We therefore label each edge $E$ of $UC_{n}(\Gamma)$ with a closed edge $e$ of $\Gamma$.
    \item A collection of $m$ edges of $UC_{n}(\Gamma)$ with a common endpoint span an $m$--cube if their labels are pairwise disjoint.
\end{itemize}
\end{construction}

In the case $n=2$, we can use this procedure to give a simple algorithm for drawing the cube complex $UC_2(\Gamma)$.

\vspace{3mm}

\begin{algorithm}\label{alg:comb_config_construction} Let $\Gamma$ be a finite graph with $m$ vertices, and let $n$ be a positive integer.
\begin{enumerate}
    \item Denote the vertices of $\Gamma$ by $v_{1}, \dots, v_{m}$.
    \item Draw $m^2$ vertices, arranged in an $m \times m$ grid.
    \item Remove the diagonal and the upper triangle. The remaining lower triangle is the $0$--skeleton $UC_2(\Gamma)^{(0)}$, with the vertex $w_{ij}$ at the $(i,j)$ coordinate corresponding to the configuration of one particle at vertex $v_i$ and one particle at vertex $v_j$.
    \item Add an edge between $w_{ij}$ and $w_{kl}$ if either $i=k$ and $\{v_j,v_l\} \in E(\Gamma)$, or $j=l$ and $\{v_i,v_k\} \in E(\Gamma)$, or $i=l$ and $\{v_j,v_k\} \in E(\Gamma)$, or $j=k$ and $\{v_i,v_l\} \in E(\Gamma)$. Label the edge between $w_{ij}$ and $w_{kl}$ with the corresponding edge of $E(\Gamma)$.
    \item For each vertex $w_{ij}$ and every pair of edges adjacent to $w_{ij}$ labeled by $\{v_i, v_k\}$ and $\{v_j,v_l\}$ for some $k \neq l$, add a $2$--cube whose $1$--skeleton is the $4$--cycle $w_{ij},w_{kj},w_{kl},w_{il}$.
\end{enumerate}
\end{algorithm}

An example is given in Fig. \ref{fig:counterexample} below. Note, the darker shaded regions are the tori arising from products of disjoint $3$--cycles in $\Delta$.

\begin{figure}[ht]
     \centering
     \def\svgscale{0.45}
     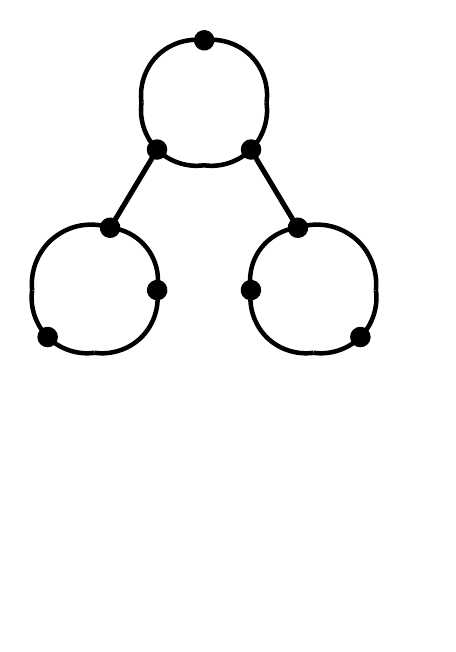
     \def\svgscale{0.675}
     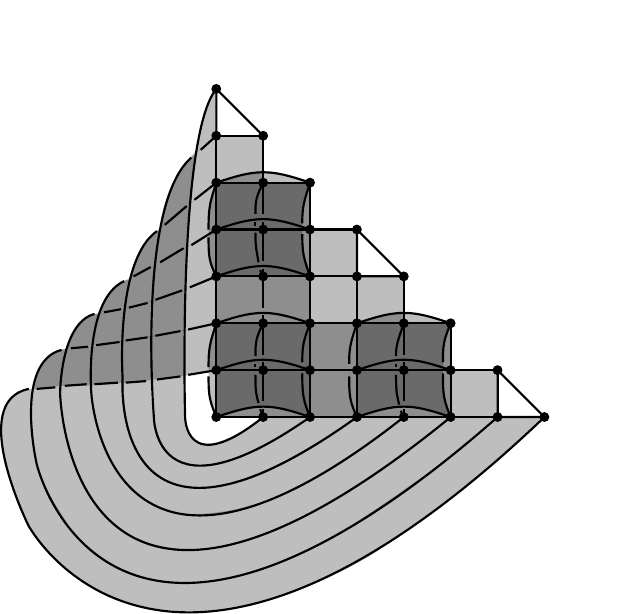
     \caption{A graph $\Delta$ and the cube complex $UC_2(\Delta)$.}
     \label{fig:counterexample}
 \end{figure}

Abrams showed that if $\Gamma$ has more than $n$ vertices, then $UC_{n}(\Gamma)$ is connected if and only if $\Gamma$ is connected. 

\begin{thm}[{\cite[Theorem 2.6]{AbramsThesis}}]\label{thm:connected_config_space}
Let $\Gamma$ be a finite graph. If $\Gamma$ has more than $n$ vertices, then $UC_{n}(\Gamma)$ is connected if and only if $\Gamma$ is connected.
\end{thm}

By analysing Abrams' proof, we are able to extract additional information regarding how connected components of $\Gamma$ give rise to connected components of $UC_n(\Gamma)$.

\begin{cor}\label{cor:no_of_comps_of_UC}
Let $\Gamma$ be a finite graph with $k$ connected components. If each component of $\Gamma$ has at least $n$ vertices, then $UC_n(\Gamma)$ has ${n+k-1 \choose k-1}$ connected components, corresponding to the number of ways to partition the $n$ particles into the $k$ connected components of $\Gamma$. 
\end{cor}

\begin{proof}
Abrams' proof of Theorem \ref{thm:connected_config_space} proceeds by showing that if $x,y \in UC_n(\Gamma)$ are two configurations where $n-1$ of the particles are in the same locations in $x$ and $y$, and the remaining particle lies in different components of $\Gamma$ for $x$ and $y$, then there is no path in $UC_n(\Gamma)$ connecting $x$ and $y$. Conversely, Abrams shows that if the remaining particle lies in the same component of $\Gamma$ for $x$ and $y$, then there is a path in $UC_n(\Gamma)$ connecting $x$ and $y$. By concatenating such paths, we see that two arbitrary configurations $x,y \in UC_n(\Gamma)$ are connected by a path if and only if $x$ and $y$ have the same number of particles in each connected component of $\Gamma$. Thus, the number of connected components of $UC_n(\Gamma)$ is precisely the number of ways the $n$ particles can be partitioned into the $k$ components of $\Gamma$. Using the stars and bars technique, we therefore see that $UC_n(\Gamma)$ has ${n+k-1 \choose k-1}$ connected components.
\end{proof}

Furthermore, Abrams showed that if we subdivide edges of $\Gamma$ sufficiently (i.e. add $2$--valent vertices to the middle of edges) to give a new graph $\Gamma'$, then $UC_{n}^{\text{top}}(\Gamma')$ deformation retracts onto $UC_{n}(\Gamma')$ {\cite[Theorem 2.1]{AbramsThesis}}. As $UC_{n}^{\text{top}}(\Gamma')$ does not distinguish vertices from other points on the graph, we have $UC_{n}^{\text{top}}(\Gamma')$ = $UC_{n}^{\text{top}}(\Gamma)$, implying that $B_{n}(\Gamma,S) \cong RB_{n}(\Gamma',S)$. This allows us to consider $B_{n}(\Gamma,S)$ as the fundamental group of the cube complex $UC_{n}(\Gamma')$.  

Prue and Scrimshaw later improved upon the constants in Abrams' result to give the following theorem {\cite[Theorem 3.2]{PrueScrimshaw}}.

\begin{thm}[\cite{AbramsThesis},\cite{PrueScrimshaw}]\label{thm:config_space_retraction}
Let $n \in \mathbb{N}$ and let $\Gamma$ be a finite graph with at least $n$ vertices. The unordered topological configuration space $UC_{n}^{\text{top}}(\Gamma)$ deformation retracts onto the unordered combinatorial configuration space $UC_{n}(\Gamma)$ (and in particular $RB_n(\Gamma) \cong B_n(\Gamma)$) if the following conditions hold.
\begin{enumerate}
    \item Every path between distinct vertices of $\Gamma$ of valence $\neq 2$ has length at least $n-1$.
    \item Every homotopically essential loop in $\Gamma$ has length at least $n+1$.
\end{enumerate}
\end{thm}

Note that Prue and Scrimshaw's version of the theorem only deals with the case where $\Gamma$ is connected. However, the disconnected case follows easily by deformation retracting each connected component of $UC_{n}(\Gamma)$, noting that each component can be expressed as a product of cube complexes $UC_{k}(\Lambda)$, where $k \leq n$ and $\Lambda$ is a connected component of $\Gamma$. In fact, we have the following more general result about configuration spaces of disconnected graphs, due to Genevois.

\begin{lem}[{\cite[Lemma 3.5]{Gen_braids}}]\label{lem:braid_is_product_of_connected_components}
Let $n>1$, let $\Gamma$ be a finite graph, and suppose $\Gamma = \Gamma_{1} \sqcup \Gamma_{2}$. Then
\[ UC_{n}(\Gamma) \simeq \bigsqcup_{k=0}^{n} UC_{k}(\Gamma_{1}) \times UC_{n-k}(\Gamma_{2}). \]
Moreover, if $S \in UC_{n}(\Gamma)$ has $k$ particles in $\Gamma_{1}$ and $n-k$ particles in $\Gamma_{2}$, then \[RB_{n}(\Gamma,S) \cong RB_{k}(\Gamma_{1},S\cap\Gamma_{1}) \times RB_{n-k}(\Gamma_{2},S\cap\Gamma_{2}),\] where $S \cap \Gamma_{1}$ (resp. $S \cap \Gamma_{2}$) denotes the configuration of the $k$ (resp. $n-k$) particles of $S$ lying in $\Gamma_{1}$ (resp. $\Gamma_{2}$).
\end{lem}

\begin{remark}
Note, Genevois states this result in terms of $B_n(\Gamma,S)$ rather than $RB_n(\Gamma,S)$. However, the proof proceeds by showing it is true for $RB_n(\Gamma,S)$ and then applying Theorem \ref{thm:config_space_retraction} after subdividing edges of $\Gamma$ sufficiently. 
\end{remark}

Another foundational result of Abrams states that the cube complex $UC_{n}(\Gamma)$ is non-positively curved \cite[Theorem 3.10]{AbramsThesis}. Furthermore, Genevois proved that $UC_{n}(\Gamma)$ admits a \emph{special coloring} \cite[Proposition 3.7]{Gen_braids}. We shall omit the details of his theory of special colorings, and direct the reader to \cite{Gen_rel_hyp_special} for further details. The key result is that a cube complex $X$ admits a special coloring if and only if there exists a special cube complex $Y$ such that $Y^{(2)}=X^{(2)}$ \cite[Lemma 3.2]{Gen_rel_hyp_special}. Furthermore, Genevois constructs $Y$ by taking $X^{(2)}$ and inductively attaching $m$--cubes $C$ whenever a copy of $C^{(m-1)}$ is present in the complex, for $m \geq 3$; this ensures non-positive curvature of $Y$. Since in our case $X=UC_{n}(\Gamma)$ is already non-positively curved, this means $Y=X$. Thus, $B_{n}(\Gamma)$ is the fundamental group of the special cube complex $UC_{n}(\Gamma')$. Crisp and Wiest also showed this indirectly by proving that every graph braid group embeds in a right-angled Artin group \cite[Theorem 2]{CrispWiest}.

In summary, we have the following result.

\begin{cor}[Graph braid groups are special; \cite{AbramsThesis,CrispWiest,Gen_braids,Gen_rel_hyp_special}]\label{cor:braids_are_cubical}
Let $n>1$ and let $\Gamma$ be a finite, connected graph. Then $B_{n}(\Gamma) \cong RB_{n}(\Gamma')$, where $\Gamma'$ is obtained from $\Gamma$ by subdividing edges. In particular, $B_{n}(\Gamma)$ is the fundamental group of the connected compact special cube complex $UC_{n}(\Gamma')$.
\end{cor}

Furthermore, Genevois provides a useful combinatorial criterion for detecting flats in graph braid groups. This is summarized by the following two lemmas.

\begin{lem}[Subgraphs induce subgroups {\cite[Proposition 3.10]{Gen_braids}}]\label{lem:subgraphs_induce_subgroups}
Let $\Gamma$ be a finite, connected graph and let $n\geq 1$. If $\Lambda$ is a proper subgraph of $\Gamma$ and $|V(\Gamma) \smallsetminus V(\Lambda)| \geq n-m$, then $RB_m(\Lambda,S)$ embeds as a subgroup of $RB_n(\Gamma)$ for all $m \leq n$ and all base points $S \in UC_m(\Lambda)$.
\end{lem}

Note, the condition $|V(\Gamma) \smallsetminus V(\Lambda)| \geq n-m$ is required in the reduced case to guarantee that there are sufficiently many vertices outside of $\Lambda$ on which $n-m$ particles can be fixed while the remaining $m$ particles move around $\Lambda$.

\begin{lem}[Criterion for infinite diameter {\cite[Lemma 4.3]{Gen_braids}}]\label{lem:infinite_diameter_braid_criterion}
Let $\Gamma$ be a finite graph, let $n \geq 1$, and let $S$ be a vertex of $UC_n(\Gamma)$, so that $S$ is a subset of $V(\Gamma)$ with size $|S|=n$. Then $B_n(\Gamma,S)$ has infinite cardinality if and only if one of the following holds; otherwise, $B_n(\Gamma,S)$ is trivial.
\begin{itemize}
    \item $n=1$ and the connected component of $\Gamma$ containing $S$ has a cycle subgraph;
    \item $n \geq 2$ and either $\Gamma$ has a connected component whose intersection with $S$ has cardinality at least $1$ and which contains a cycle subgraph, or $\Gamma$ has a connected component whose intersection with $S$ has cardinality at least $2$ and which contains a vertex of valence at least $3$.
\end{itemize}
\end{lem}

Combining Lemmas \ref{lem:braid_is_product_of_connected_components}, \ref{lem:subgraphs_induce_subgroups}, and \ref{lem:infinite_diameter_braid_criterion} proves one direction of the following theorem. The other direction is proven in \cite[Proof of Theorem 4.1]{Gen_braids}.

\begin{thm}\label{thm:Z2_characterisation}
Let $\Gamma$ be a finite connected graph and let $n \geq 1$. The graph braid group $B_n(\Gamma)$ contains a $\mathbb{Z}^2$ subgroup if and only if one of the following holds.
\begin{itemize}
    \item $n=2$ and $\Gamma$ contains two disjoint cycle subgraphs;
    \item $n=3$ and $\Gamma$ contains either two disjoint cycle subgraphs or a cycle and a vertex of valence at least $3$ disjoint from it;
    \item $n \geq 4$ and $\Gamma$ contains either two disjoint cycle subgraphs, or a cycle and a vertex of valence at least $3$ disjoint from it, or two vertices of valence at least $3$.
\end{itemize}
\end{thm}

One may also produce similar theorems characterising $\mathbb{Z}^m$ subgroups for any $m \geq 1$.

\subsection{\textit{Graphs of groups}}\label{section:background_graphs_of_groups}

A \emph{graph of groups} is a system $(\mc{G},\Lambda)$ consisting of: 
\begin{itemize}
    \item an oriented connected graph $\Lambda$;
    \item a group $G_v$ for each vertex $v \in V(\Lambda)$ (called a \emph{vertex group}) and a group $G_e$ for each edge $e \in E(\Lambda)$ (called an \emph{edge group});
    \item monomorphisms $\phi_e^-: G_e \rightarrow G_{o(e)}$ and $\phi_e^+: G_e \rightarrow G_{t(e)}$, where $o(e)$ and $t(e)$ denote the initial and terminal vertices of $e$, respectively.
\end{itemize}

We recall two ways of defining a graph of groups decomposition of a group $G$. 

\begin{defn}[Fundamental group of $(\mc{G},\Lambda)$]
Let $(\mc{G},\Lambda)$ be a graph of groups and let $T$ be a spanning tree for $\Lambda$. The \emph{fundamental group of} $(\mc{G},\Lambda)$ \emph{based at} $T$, denoted $\pi_1(\mc{G},\Lambda;T)$, is the quotient of the free product 
\[\left(\ast_{v \in V(\Lambda)}G_v\right) \ast F(E(\Lambda))\]
by 
\[ \nsub{\{e^{-1} \phi_e^-(g) e \phi_e^+(g)^{-1} \,\,|\,\, e \in E(\Lambda), g \in G_e\} \cup \{e \,\,|\,\, e \in E(T)\}}. \]

We say that a group $G$ \emph{decomposes as a graph of groups} $(\mc{G},\Lambda)$ if there exists a spanning tree $T$ of $\Lambda$ such that $G \cong \pi_1(\mc{G},\Lambda;T)$.
\end{defn}

\begin{remark}\label{graphs_of_gps_and_spanning_trees}
Note that $\pi_1(\mc{G},\Lambda;T)$ can also be expressed as the quotient of 
\[\pi_1(\mc{G},T;T) \ast F(E(\Lambda)\smallsetminus E(T))\] 
by the normal subgroup generated by elements of the form $e^{-1}\phi_e^-(g)e\phi_e^+(g)^{-1}$, where $e \in E(\Lambda)\smallsetminus E(T)$ and $g \in G_e$; this is immediate from the definition.
\end{remark}

\begin{remark}
The group $\pi_1(\mc{G},\Lambda;T)$ does not depend on the choice of spanning tree $T$ \cite[Proposition I.20]{Serre_Trees}. We therefore often call $\pi_1(\mc{G},\Lambda;T)$ simply the \emph{fundamental group of} $(\mc{G},\Lambda)$ and denote it $\pi_1(\mc{G},\Lambda)$.
\end{remark}

We may also define a graph of groups decomposition as the fundamental group of a graph of spaces. Recall that a \emph{graph of pointed CW complexes} is a system $(\mc{X},\Lambda)$ consisting of:
\begin{itemize}
    \item an oriented connected graph $\Lambda$;
    \item a pointed path-connected CW complex $(X_v,x_v)$ for each vertex $v \in V(\Lambda)$ (called a \emph{vertex space}) and $(X_e,x_e)$ for each edge $e \in E(\Lambda)$ (called an \emph{edge space});
    \item cellular maps $p_e^-:(X_e,x_e)\rightarrow (X_{o(e)},x_{o(e)})$ and $p_e^+:(X_e,x_e)\rightarrow (X_{t(e)},x_{t(e)})$.
\end{itemize}
In particular, when each of the induced maps $p_{e\#}^-: \pi_1(X_e,x_e) \rightarrow \pi_1(X_{o(e)},x_{o(e)})$ and $p_{e\#}^+: \pi_1(X_e,x_e) \rightarrow \pi_1(X_{t(e)},x_{t(e)})$ is a monomorphism, we obtain a graph of groups $(\mc{G},\Lambda)$ where $G_v = \pi_1(X_v,x_v)$, $G_e = \pi_1(X_e,x_e)$, and $\phi_e^\pm = p_{e\#}^\pm$.

\begin{defn}[Total complex]
The \emph{total complex} associated to the graph of pointed CW complexes $(\mc{X},\Lambda)$, denoted $\Tot(\mc{X},\Lambda)$, is obtained by taking the disjoint union of the vertex spaces $X_v$ and the products $X_e \times [-1,1]$ of the edge spaces with intervals, then gluing these spaces using maps $p_e: X_e \times \{-1,1\} \rightarrow X_{o(e)} \sqcup X_{t(e)}$ defined by $p_e(x,\pm 1) = p_e^\pm(x)$.
\end{defn}

\begin{prop}[{\cite[Proposition 6.2.2]{Geoghegan_Top_Methods_in_Group_Thy}}]
Suppose each $p_{e\#}^\pm$ is a monomorphism. Then $\pi_1(\Tot(\mc{X},\Lambda),x_v)$ is isomorphic to $\pi_1(\mc{G},\Lambda)$ for any choice of $v \in V(\Lambda)$.
\end{prop}

\begin{remark}\label{rem:graph_of_gps_decomp_of_a_cube_cx}
We may apply this to a special cube complex $X$ by cutting $X$ along a collection of disjoint hyperplanes $\{H_e\}_{e \in E}$; that is, remove the open carrier of each $H_e$ to obtain a cube complex $X' \subseteq X$. Denote the connected components of $X'$ by $\{X_v\}_{v\in V}$ and let $X_e = H_e$. Since hyperplanes of $X$ are two-sided by Definition \ref{defn:special}(1), the open carrier of $H_e$ is homeomorphic to $X_e \times (-1/2,1/2)$. In particular, we may fix an orientation on $H_e$ for each $e \in E$, denoting the combinatorial hyperplanes corresponding to $X_e \times \{\pm 1/2\}$ by $H_e^{\pm}$ as in Notation \ref{not:comb_hyp}.

Since hyperplanes of $X$ do not self-intersect by Definition \ref{defn:special}(2), for each $e \in E$ the combinatorial hyperplanes $H_e^\pm$ lie in some connected components $X_v$ and $X_w$ of $X'$, respectively (we may have $v=w$). Thus, to each $e \in E$ we may associate the pair $(v,w) \in V \times V$. That is, $V$ and $E$ define the vertex set and edge set of some graph $\Lambda$, which is connected since $X$ is connected. Furthermore, for each $e \in E$ the orientation on $H_e$ induces an orientation on the edge $e$ of $\Lambda$. Since hyperplanes of $X$ do not directly self-osculate by Definition \ref{defn:special}(3), we may define maps $p_e^\pm$ by identifying $X_e \times \{\pm 1/2\}$ with the combinatorial hyperplanes $H_e^\pm$ in $X_v$ and $X_w$. Combinatorial hyperplanes of non-positively curved cube complexes are locally convex by Proposition \ref{prop:local_convexity_of_comb_hyps}, so $H_e^\pm$ are locally convex in $X$ and hence also in $X_v$ and $X_w$. Thus, the maps $p_e^\pm$ are $\pi_1$--injective by \cite[Lemma 2.11]{HaglundWise_Special}. It follows that $p_e^\pm$ induce monomorphisms $p_{e\#}^\pm$ on the fundamental groups. This data therefore defines a graph of spaces $(\mc{X},\Lambda)$ such that $X = \Tot(\mc{X},\Lambda)$.
\end{remark}

\section{Decomposing Graph Braid Groups as Graphs of Groups}\label{sec:decomposing_gbgs}

A graph braid group $B_n(\Gamma)$ may be decomposed as a graph of groups by cutting $\Gamma$ along a collection of edges that share a common vertex. In this section we show how to construct such a graph of groups decomposition. We begin by noting a result of Genevois, given below. Recall that we consider vertices of $UC_n(\Gamma)$ to be subsets of $V(\Gamma)$, as in Construction \ref{constr:Gen_UC}. 

\begin{lem}[{\cite[Lemma 3.6]{Gen_braids}}]\label{lem:Gen_hyp_labelling}
Let $E_1$ and $E_2$ be two edges of $UC_n(\Gamma)$ with endpoints $S_1, S'_1$ and $S_2, S'_2$, respectively. Furthermore, let $e_i$ be the (closed) edge of $\Gamma$ labeling $E_i$ for $i=1,2$. Then $E_1$ and $E_2$ are dual to the same hyperplane of $UC_n(\Gamma)$ if and only if $e_1 = e_2 =: e$ and $S_1 \cap (\Gamma \smallsetminus e)$, $S_2 \cap (\Gamma \smallsetminus e)$ belong to the same connected component of $UC_{n-1}(\Gamma \smallsetminus e)$.
\end{lem}

\begin{convention}
Note, when we write $\Gamma \smallsetminus \Omega$ for some subgraph $\Omega \subseteq \Gamma$, we mean the induced subgraph of $\Gamma$ on the vertex set $V(\Gamma) \smallsetminus V(\Omega)$.
\end{convention}

\begin{remark}[Characterizing combinatorial hyperplanes]\label{rem:characterising_comb_hyps}
Lemma \ref{lem:Gen_hyp_labelling} implies that each hyperplane $H$ of $UC_n(\Gamma)$ may be labeled by an edge $e$ of $\Gamma$. Moreover, the combinatorial hyperplanes $H^\pm$ associated to $H$ are two subcomplexes of $UC_n(\Gamma)$ isomorphic to the connected component of $UC_{n-1}(\Gamma \smallsetminus e)$ containing $S \cap (\Gamma \smallsetminus e)$ for some (hence any) $S \in (H^\pm)^{(0)}$. These two isomorphisms are obtained by fixing a particle at the two endpoints of $e$, respectively. 
\end{remark}

\begin{remark}[Counting hyperplanes]\label{rem:counting_hyperplanes}
Another immediate consequence of Lemma \ref{lem:Gen_hyp_labelling} is that the number of hyperplanes of $UC_n(\Gamma)$ labeled by $e \in E(\Gamma)$ is equal to the number of connected components of $UC_{n-1}(\Gamma \smallsetminus e)$. If we let $k_e$ denote the number of connected components of $\Gamma \smallsetminus e$, then Corollary \ref{cor:no_of_comps_of_UC} implies there are ${n+k_e-2 \choose k_e-1}$ hyperplanes of $UC_n(\Gamma)$ labeled by $e$. In other words, each hyperplane labeled by $e$ corresponds to fixing one particle in $e$ and partitioning the remaining $n-1$ particles among the connected components of $\Gamma \smallsetminus e$.
\end{remark}

We may combine Lemma \ref{lem:Gen_hyp_labelling} with Remark \ref{rem:graph_of_gps_decomp_of_a_cube_cx} to obtain graph of groups decompositions of a graph braid group where the vertex groups and edge groups are braid groups on subgraphs. The following theorem may be viewed as a strengthening of \cite[Proposition 4.6]{Gen_braids}.

\begin{thm}\label{thm:graphs_of_graph_braid_groups}
Let $\Gamma$ be a finite, connected, oriented graph (not necessarily satisfying the hypotheses of Theorem \ref{thm:config_space_retraction}), let $n \geq 2$, let $m \geq 1$, and let $e_1, \dots, e_m$ be distinct edges of $\Gamma$ sharing a common vertex $v$. 
For each $i \in \{1,\dots,m\}$, let $\mc{H}_i$ be the collection of hyperplanes of $UC_n(\Gamma)$ that are dual to edges of $UC_n(\Gamma)$ labeled by $e_i$, with orientations induced by the orientation of $e_i$, and let $\mc{H} = \mc{H}_1 \cup \dots \cup \mc{H}_m$. The reduced graph braid group $RB_n(\Gamma)$ decomposes as a graph of groups $(\mc{G},\Lambda)$, where:
\begin{itemize}
    \item $V(\Lambda)$ is the collection of connected components of $UC_n(\Gamma \smallsetminus (\mathring{e}_1\cup\dots\cup\mathring{e}_m))$;
    \item $E(\Lambda) = \mc{H}$, where $H \in E(\Lambda)$ has initial and terminal vertices $K \in V(\Lambda)$ and $L \in V(\Lambda)$ (not necessarily distinct) if $H^-$ and $H^+$ lie in the components $K$ and $L$ of $UC_n(\Gamma \smallsetminus (\mathring{e}_1\cup\dots\cup\mathring{e}_m))$, respectively;
    \item for each $K \in V(\Lambda)$, we have $G_K = RB_n(\Gamma \smallsetminus (\mathring{e}_1\cup\dots\cup\mathring{e}_m), S_K)$ for some (equivalently any) basepoint $S_K \in K$;
    \item for each $H_i \in \mc{H}_i \subseteq E(\Lambda)$, we have $G_{H_i} = RB_{n-1}(\Gamma \smallsetminus e_i, S_{H_i} \cap (\Gamma \smallsetminus e_i))$, for some (equivalently any) vertex $S_{H_i}$ of one of the combinatorial hyperplanes $H^\pm_i$;
    \item for each edge $H \in E(\Lambda)$ joining vertices $K,L \in V(\Lambda)$, the monomorphisms $\phi^\pm_{H}$ are induced by the inclusion maps of the combinatorial hyperplanes $H^\pm$ into $K$ and $L$.
\end{itemize}
\end{thm}

\begin{warning}
One must be careful when computing vertex groups and edge groups in such a graph of groups decomposition, as removing edges from $\Gamma$ may introduce new vertices of valence $\neq 2$. This means that even if we assume $\Gamma$ satisfies the hypotheses of Theorem \ref{thm:config_space_retraction}, they may no longer hold when edges are removed. Thus, the reduced graph braid groups arising as vertex groups and edge groups may not be isomorphic to the corresponding graph braid groups even when $RB_n(\Gamma) \cong B_n(\Gamma)$.
\end{warning}

\begin{proof}[Proof of Theorem \ref{thm:graphs_of_graph_braid_groups}]
By Lemma \ref{lem:Gen_hyp_labelling} and Construction \ref{constr:Gen_UC}, two hyperplanes of $UC_n(\Gamma)$ cross only if they are dual to edges of $UC_n(\Gamma)$ labeled by disjoint edges of $\Gamma$. Since the edges $e_1, \dots, e_m$ share a common vertex, the hyperplanes in $\mc{H}$ are therefore pairwise disjoint.

By Remark \ref{rem:graph_of_gps_decomp_of_a_cube_cx}, we may cut along $\mc{H}$ to obtain a graph of spaces decomposition $(\mc{X},\Lambda)$ of $UC_n(\Gamma)$ (recall that $UC_n(\Gamma)$ is special by Corollary \ref{cor:braids_are_cubical}). Moreover, by cutting along $\mc{H}$, we are removing precisely the edges of $UC_n(\Gamma)$ that are labeled by $e_i$ for some $i$. Thus, the cube complex we are left with is $UC_n(\Gamma \smallsetminus (\mathring{e}_1\cup\dots\cup\mathring{e}_m))$. The vertex spaces of $(\mc{X},\Lambda)$ are therefore the connected components of $UC_n(\Gamma \smallsetminus (\mathring{e}_1\cup\dots\cup\mathring{e}_m))$; that is, we may define $V(\Lambda)$ to be the collection of connected components of $UC_n(\Gamma \smallsetminus (\mathring{e}_1\cup\dots\cup\mathring{e}_m))$ and take $X_K = K$ for each $K \in V(\Lambda)$. In particular, the vertex groups are
\[G_K = \pi_1(X_K) = \pi_1(K) \cong \pi_1(UC_n(\Gamma \smallsetminus (\mathring{e}_1\cup\dots\cup\mathring{e}_m)),S_K) = RB_n(\Gamma \smallsetminus (\mathring{e}_1\cup\dots\cup\mathring{e}_m), S_K),\] 
where $S_K$ is any point in $K$.

Remark \ref{rem:graph_of_gps_decomp_of_a_cube_cx} further tells us that $\Lambda$ has an edge from $K \in V(\Lambda)$ to $L \in V(\Lambda)$ (not necessarily distinct) for every hyperplane $H \in \mc{H}$ that has combinatorial hyperplanes $H^-$ in $K$ and $H^+$ in $L$, and the associated edge space is $H$. Thus, we may define $E(\Lambda) = \mc{H}$ and take $X_H = H$ for each $H \in E(\Lambda)$. Furthermore, Lemma \ref{lem:Gen_hyp_labelling} tells us that the combinatorial hyperplanes $H_i^\pm$ associated to $H_i \in \mc{H}_i$ are each isomorphic to the connected component of $UC_{n-1}(\Gamma \smallsetminus e_i)$ containing $S_{H_i} \cap (\Gamma \smallsetminus e_i)$ for some (hence any) $S_{H_i} \in (H_i^\pm)^{(0)}$. That is, the edge groups are
\[G_{H_i} = \pi_1(X_{H_i}) \cong \pi_1(H_i^\pm) \cong \pi_1(UC_{n-1}(\Gamma \smallsetminus e_i),S_{H_i}\cap(\Gamma\smallsetminus e_i)) = RB_{n-1}(\Gamma \smallsetminus e_i, S_{H_i} \cap (\Gamma \smallsetminus e_i)). \]

Finally, Remark \ref{rem:graph_of_gps_decomp_of_a_cube_cx} tells us if $H \in E(\Lambda)$ connects $K,L \in V(\Lambda)$, the cellular maps $p_{H}^\pm$ are obtained by identifying $X_H \times \{\pm 1/2\} = H \times \{\pm 1/2\}$ with the two combinatorial hyperplanes $H^\pm$ lying in $K$ and $L$ via the inclusion maps. Combinatorial hyperplanes of non-positively curved cube complexes are locally convex by Proposition \ref{prop:local_convexity_of_comb_hyps}, so $H^\pm$ are locally convex in $UC_n(\Gamma)$ and hence also in $K$ and $L$. Thus, the maps $p_H^\pm$ are $\pi_1$--injective by \cite[Lemma 2.11]{HaglundWise_Special}. In particular, each of the induced maps $p^\pm_{H\#}$ is a monomorphism; denote these monomorphisms by $\phi_H^\pm$.
\end{proof}

In order to use this theorem to compute graph braid groups, we need to be able to determine when two vertices are connected by an edge, and also count the number of edges between each pair of vertices. 

We first fix some notation. Recall that the edges $e_i$ in Theorem \ref{thm:graphs_of_graph_braid_groups} share a common vertex $v$; let $v_i$ denote the other vertex of $e_i$. Let $\Gamma_v$ and $\Gamma_i$ denote the connected components of $\Gamma \smallsetminus (\mathring{e}_1\cup\dots\cup\mathring{e}_m)$ containing $v$ and $v_i$, respectively. 

\begin{prop}[Determining adjacency]\label{prop:determining_adjacency}
Let $(\mc{G},\Lambda)$ be the graph of groups decomposition of $RB_n(\Gamma)$ described in Theorem \ref{thm:graphs_of_graph_braid_groups}. Two vertices $K,L \in V(\Lambda)$ are connected by an edge in $\mc{H}_i \subseteq E(\Lambda)$ if and only if the corresponding partitions given by Corollary \ref{cor:no_of_comps_of_UC} differ by moving a single particle from $\Gamma_v$ to $\Gamma_i$ (not necessarily distinct) along $e_i$.
\end{prop}

\begin{proof}
Let $(\mc{G},\Lambda)$ be a graph of groups decomposition of $RB_n(\Gamma)$ arising from Theorem \ref{thm:graphs_of_graph_braid_groups}, and let $K,L \in V(\Lambda)$; that is, $K$ and $L$ are connected components of $UC_n(\Gamma \smallsetminus (\mathring{e}_1\cup\dots\cup\mathring{e}_m))$. Recall that by Corollary \ref{cor:no_of_comps_of_UC}, connected components of $UC_n(\Gamma \smallsetminus (\mathring{e}_1\cup\dots\cup\mathring{e}_m))$ correspond to partitions of the $n$ particles among the connected components of $\Gamma \smallsetminus (\mathring{e}_1\cup\dots\cup\mathring{e}_m)$. Furthermore, $K,L \in V(\Lambda)$ are connected by an edge $H_i \in \mc{H}_i \subseteq E(\Lambda)$ if and only if $H_i$ has a combinatorial hyperplane in $K$ and another in $L$. By Remark \ref{rem:characterising_comb_hyps}, the partitions corresponding to $K$ and $L$ differ by moving a single particle from $\Gamma_v$ to $\Gamma_i$ along $e_i$.
\end{proof}

For each $K,L \in V(\Lambda)$ connected by some edge in $\mc{H}_i$, let $n^K_v, n^K_i$ and $n^L_v, n^L_i$ denote the number of particles in $\Gamma_v, \Gamma_i$ in the partitions corresponding to $K$ and $L$, respectively. Let $n_v = \max\{n^K_v,n^L_v\}-1$ and $n_i = \max\{n^K_i,n^L_i\}-1$. Let $k_v$, $k_i$, and $k_{v,i}$ denote the number of connected components of $\Gamma_v \smallsetminus v$, $\Gamma_i \smallsetminus v_i$, and $\Gamma_v \smallsetminus (v \cup v_i)$, respectively. We have the following result.

\begin{prop}[Counting edges]\label{prop:counting_edges}
Let $(\mc{G},\Lambda)$ be the graph of groups decomposition of $RB_n(\Gamma)$ described in Theorem \ref{thm:graphs_of_graph_braid_groups}, and suppose $K,L \in V(\Lambda)$ are connected by some edge in $\mc{H}_i$. Let $l_i$ denote the number of edges in $\mc{H}_i \subseteq E(\Lambda)$ connecting $K$ and $L$.
\begin{itemize}
\item If $\Gamma_v = \Gamma_i$, then $l_i$ is equal to the number of ways of partitioning the $n_v$ particles among the $k_{v,i}$ components of $\Gamma_v \smallsetminus (v \cup v_i)$. If each component of $\Gamma_v \smallsetminus (v \cup v_i)$ has at least $n_v$ vertices, then $l_i = {n_v+k_{v,i}-1 \choose k_{v,i}-1}$.
\item If $\Gamma_v \neq \Gamma_i$, then $l_i$ is equal to the number of ways of partitioning the $n_v$ particles among the $k_{v}$ components of $\Gamma_v \smallsetminus v$ and the $n_i$ particles among the $k_i$ components of $\Gamma_i \smallsetminus v_i$. If each component of $\Gamma_v \smallsetminus v$ has at least $n_v$ vertices and each component of $\Gamma_i \smallsetminus v_i$ has at least $n_i$ vertices, then $l_i = {n_v+k_v-1 \choose k_v-1}{n_i+k_i-1 \choose k_i-1}$.
\end{itemize}
Furthermore, the total number of edges of $\Lambda$ connecting the vertices $K$ and $L$ is equal to $\sum_{j \in I}l_j$, where $I = \{1 \leq j \leq m \,|\, \Gamma_j = \Gamma_i\}$.
\end{prop}

\begin{proof}
Let $(\mc{G},\Lambda)$ be a graph of groups decomposition of $RB_n(\Gamma)$ arising from Theorem \ref{thm:graphs_of_graph_braid_groups}, and suppose $K,L \in V(\Lambda)$ are connected by an edge in $\mc{H}_i$. By Proposition \ref{prop:determining_adjacency}, the partitions corresponding to $K$ and $L$ given by Corollary \ref{cor:no_of_comps_of_UC} differ by moving a single particle from $\Gamma_v$ to $\Gamma_i$ along $e_i$. Applying Remark \ref{rem:counting_hyperplanes}, the number of edges in $\mc{H}_i$ connecting $K$ and $L$ is then computed by taking the partition corresponding to $K$, fixing this particle in $e_i$, and then further partitioning the remaining particles in $\Gamma_v \cup \Gamma_i$ among the connected components of $(\Gamma_v \cup \Gamma_i) \smallsetminus (v \cup v_i)$. If $\Gamma_v = \Gamma_i$, then the number of such partitions is equal to ${n_v+k_{v,i}-1 \choose k_{v,i}-1}$ by Corollary \ref{cor:no_of_comps_of_UC}; if $\Gamma_v \neq \Gamma_i$, then the number of such partitions is equal to ${n_v+k_v-1 \choose k_v-1}{n_i+k_i-1 \choose k_i-1}$. 

It then remains to sum over all $j$ such that $K$ and $L$ are connected by an edge in $\mc{H}_j$. Note that if $K$ and $L$ are connected by an edge in both $\mc{H}_i$ and $\mc{H}_j$, then by Proposition \ref{prop:determining_adjacency} the corresponding partitions differ both by moving a single particle from $\Gamma_v$ to $\Gamma_i$ along $e_i$ and by moving a single particle from $\Gamma_v$ to $\Gamma_j$ along $e_j$. The only way this can happen is if $\Gamma_j = \Gamma_i$. Thus, the result follows.
\end{proof}

\subsection{\textit{Examples}}\label{sec:examples}

We devote this section to the computation of specific graph braid groups by application of Theorem \ref{thm:graphs_of_graph_braid_groups}. Along the way, we answer a question of Genevois for all but one graph \cite[Question 5.6]{Gen_braids}.

\begin{ex}\label{ex:HNN}
Consider the graphs $\Delta$ and $\Delta'$ shown in Figs. \ref{fig:example_graphs_and_cube_cxs_1} and \ref{fig:example_graphs_and_cube_cxs_2} together with their corresponding cube complexes $UC_2(\Delta)$ and $UC_2(\Delta')$, constructed by Algorithm \ref{alg:comb_config_construction}. Note that $\Delta$ is obtained from $\Delta'$ by removing the interior of the edge $e := \{v_1,v_9\}$. 

\begin{figure}[ht]
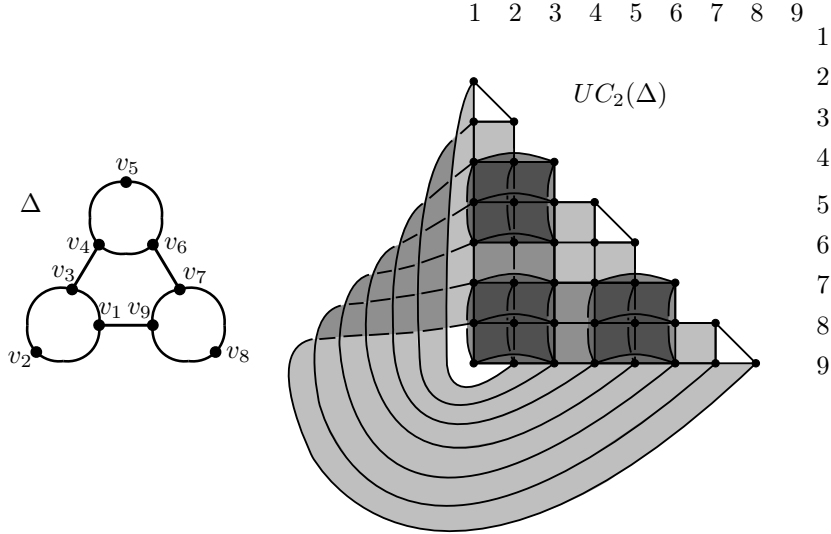

     \centering
     \def\svgscale{0.45}
     \input{big_counterexample_graph_labelled.pdf_tex} 
     \def\svgscale{0.675}
     \input{big_counterexample_labelled.pdf_tex} 
     \caption{The graph $\Delta$ and the cube complex $UC_{2}(\Delta)$.}
     \label{fig:example_graphs_and_cube_cxs_1}
 \end{figure}

 \begin{figure}[ht]
     \centering
     \def\svgscale{0.45}
     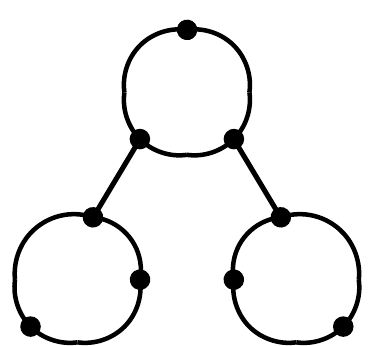 \hspace{1.2in}
     \def\svgscale{0.675}
     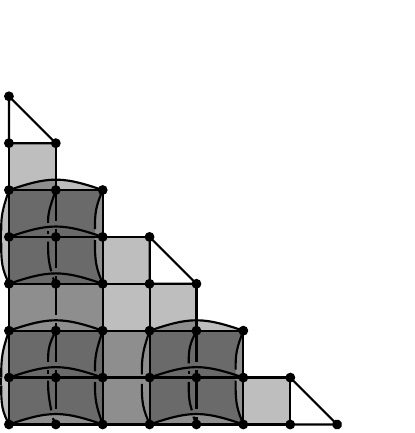
     \caption{The graph $\Delta'$ and the cube complex $UC_2(\Delta')$.}
     \label{fig:example_graphs_and_cube_cxs_2}
 \end{figure}

We may apply Theorem \ref{thm:graphs_of_graph_braid_groups} to obtain a decomposition of $B_2(\Delta)$ as a graph of groups $(\mc{G},\Lambda)$, where $V(\Lambda)$ is the collection of connected components of $UC_2(\Delta \smallsetminus \mathring{e}) = UC_2(\Delta')$. Since $\Delta'$ is connected, it follows from Theorem \ref{thm:connected_config_space} that $UC_2(\Delta')$ is also connected, hence $\Lambda$ has a single vertex with associated vertex group $RB_2(\Delta')$, which is isomorphic to $B_2(\Delta')$ since $\Delta'$ satisfies the hypotheses of Theorem \ref{thm:config_space_retraction}. 

Furthermore, $\Delta \smallsetminus e$ is connected, so $UC_1(\Delta \smallsetminus e)$ is connected and hence by Lemma \ref{lem:Gen_hyp_labelling} there is only one hyperplane of $UC_2(\Delta)$ dual to an edge labeled by $e$. That is, $\Lambda$ has a single edge with associated edge group $RB_1(\Delta \smallsetminus e) = \pi_1(\Delta \smallsetminus e) \cong \mathbb{Z}$. We therefore see that $B_2(\Delta)$ is an HNN extension of $B_2(\Delta')$; the corresponding graph of spaces decomposition of $UC_2(\Delta)$ is illustrated in Fig. \ref{fig:example_HNN}.

\begin{figure}[ht]
    \centering
    \def\svgscale{0.41}
    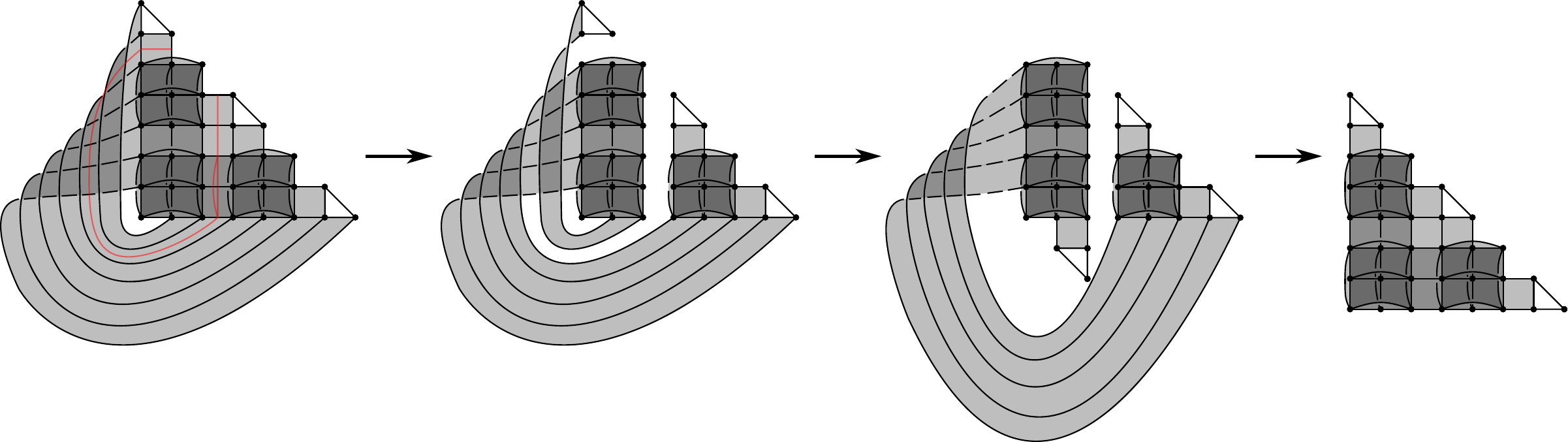
    \caption{$B_2(\Delta)$ can be realized as an HNN extension of $B_2(\Delta')$ by cutting $UC_2(\Delta)$ along a hyperplane.}
    \label{fig:example_HNN}
\end{figure}

Observe that by collapsing certain hyperplane carriers onto combinatorial hyperplanes and collapsing certain edges to points, $UC_2(\Delta')$ deformation retracts onto a wedge sum of three circles $C_1, C_2, C_3$ and three tori $T_1, T_2, T_3$, with successive tori glued along simple closed curves, denoted $\beta$ and $\gamma$. This is illustrated in Fig. \ref{fig:counterexample_def_retract}. 

\begin{figure}[ht]
     \centering
     \def\svgscale{0.675}
     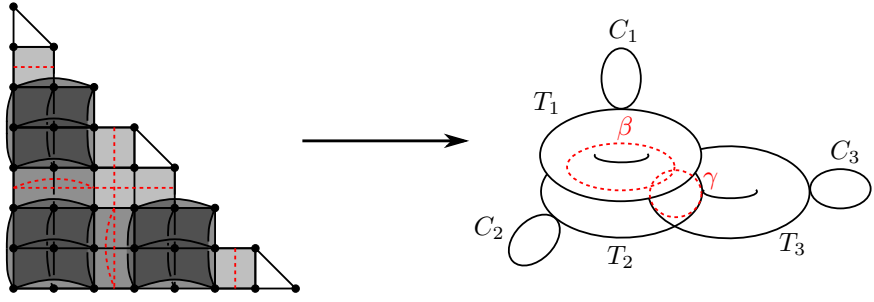
     \caption{The cube complex $UC_2(\Delta')$ deformation retracts onto three circles and three tori.}
     \label{fig:counterexample_def_retract}
 \end{figure}
 
Thus, $B_2(\Delta')$ is given by
\begin{align*}
B_2(\Delta') = \pi_1(UC_2(\Delta')) &\cong \sub{a,b,c,d,e,f,g\,|\,aba^{-1}b^{-1},\,bcb^{-1}c^{-1},\,cdc^{-1}d^{-1}}\\
&= \pi_1(T_1) \ast_{\sub{b}} \pi_1(T_2) \ast_{\sub{c}} \pi_1(T_3) \ast F_3,
\end{align*}
where $\pi_{1}(T_1) = \sub{a,b}$, $\pi_{1}(T_2) = \sub{b,c}$, $\pi_{1}(T_3) = \sub{c,d}$, and $b$ and $c$ correspond to the simple closed curves $\beta$ and $\gamma$. We then have $B_2(\Delta) = B_2(\Delta') \ast_\phi$, where $\phi$ conjugates $a$ to $d$.
\end{ex}

We now compute braid groups of radial trees and use these to build up to more complex graph braid groups. Proposition \ref{prop:radial_trees} below is also obtained by An--Maciazek \cite[Theorem 3]{AnMaciazek_geometric_presentations} using discrete Morse theory techniques.

\begin{prop}\label{prop:radial_trees}
Let $n \geq 2$ and let $R_k$ be a $k$--pronged radial tree with $k \geq 3$ (that is, $R_k$ consists of $k$ vertices of valence $1$, each joined by an edge to a central vertex $v$ of valence $k$). Then $B_n(R_k) \cong F_{M}$, where 
\[M = M(n,k) = (k-2){n+k-2 \choose k-1}-{n+k-2 \choose k-2}+1.\]
In particular, $M \geq 2$ if and only if either $n \geq 3$ or $k \geq 4$.
\end{prop}

\begin{proof}
First, subdivide each of the $k$ edges of $R_k$ by adding $n$ vertices of valence $2$ along each edge, giving a graph $R'_k$ homeomorphic to $R_k$ that satisfies the hypotheses of Theorem \ref{thm:config_space_retraction}. Let $e_1, \dots, e_k$ denote the $k$ edges of $R'_k$ sharing the central vertex $v$ and apply Theorem \ref{thm:graphs_of_graph_braid_groups} to $e_1, \dots, e_{k-2}$ to obtain a graph of groups decomposition $(\mc{G},\Lambda)$ of $B_n(R'_k)$. Note that $R'_k \smallsetminus (\mathring{e}_1 \cup \dots \cup \mathring{e}_{k-2})$ and $R'_k \smallsetminus (e_1 \cup \dots \cup e_{k-2})$ are both disjoint unions of segments, hence the vertex groups and edge groups of $(\mc{G},\Lambda)$ are all trivial. Thus, by Remark \ref{graphs_of_gps_and_spanning_trees}, $B_n(R'_k) \cong B_n(R_k)$ is a free group $F_M$ where $M = |E(\Lambda)|-|E(T)|$ for some spanning tree $T$ of $\Lambda$.

Note that $R'_k \smallsetminus (\mathring{e}_1 \cup \dots \cup \mathring{e}_{k-2})$ has $k-1$ connected components: the component $\Gamma_v$ containing $v$, and the $k-2$ components $\Gamma_i$ containing the other vertex $v_i$ of $e_i$. Thus, $\Lambda$ has ${n+k-2 \choose k-2}$ vertices by Corollary \ref{cor:no_of_comps_of_UC}, corresponding to the partitions of the $n$ vertices among the $k-1$ connected components of $R'_k \smallsetminus (\mathring{e}_1 \cup \dots \cup \mathring{e}_{k-2})$. Denote the vertex with $m$ particles in $\Gamma_v$ and $m_i$ particles in $\Gamma_i$, $i = 1, \dots, k-2$, by $K_{m,m_1,\dots,m_{k-2}}$. Then there are ${N+k-3 \choose k-3}$ vertices with $m = n-N$, corresponding to the partitions of the remaining $N$ vertices among the $k-2$ components $\Gamma_i$.

By Proposition \ref{prop:determining_adjacency}, $K_{m,m_1,\dots,m_{k-2}}$ and $K_{l,l_1,\dots,l_{k-2}}$ are connected by an edge of $\Lambda$ if and only if $|m-l|=1$, $|m_i-l_i|=1$ for some $i$, and $m_j=l_j$ for all $j \neq i$. To count the number of such pairs, fix $m$ and note that $K_{m,m_1,\dots,m_{k-2}}$ is connected by an edge to the $k-2$ vertices of the form $K_{m-1,m_1,\dots,m_i+1,\dots,m_{k-2}}$. Thus, the number of pairs of vertices that are connected by an edge of $\Lambda$ is 
\[ \sum_{N=0}^{n-1} (k-2){N+k-3 \choose k-3} = (k-2){n+k-3 \choose k-2}, \]
recalling that $\sum_{p=0}^{r}{p+q \choose q} = {r+q+1 \choose q+1}$ for any non-negative integers $p,q,r$.

Moreover, since $\Gamma_v \smallsetminus v$ has two connected components and $\Gamma_i \smallsetminus v_i$ has one connected component, Proposition \ref{prop:counting_edges} tells us the number of edges connecting the vertices $K_{m,m_1,\dots,m_{k-2}}$ and $K_{m-1,m_1,\dots,m_i+1,\dots,m_{k-2}}$ is equal to ${m-1+2-1 \choose 2-1} = m$. Thus, 
\begin{align*} 
|E(\Lambda)| &= \sum_{N=0}^{n-1} (k-2){N+k-3 \choose k-3}(n-N) \\
&= (k-2)\sum_{l=0}^{n-1}\sum_{N=0}^{n-1-l} {N+k-3 \choose k-3} \\
&= (k-2)\sum_{l=0}^{n-1} {n-l+k-3 \choose k-2} \\
&= (k-2)\sum_{L=0}^{n-1} {L+k-2 \choose k-2} \\
&= (k-2){n+k-2 \choose k-1}.
\end{align*}

Furthermore, since $T$ is a tree, we have 
\[ |E(T)| = |V(T)|-1 = |V(\Lambda)|-1 = {n+k-2 \choose k-2}-1.\] 
Thus, 
\[ M = |E(\Lambda)|-|E(T)| = (k-2){n+k-2 \choose k-1} - {n+k-2 \choose k-2} + 1.\qedhere\]
\end{proof}

\begin{prop}\label{prop:modified_radial_trees}
Let $n,k \geq 3$, $r \leq 2$, and let $R_{k,r}$ be a $k$--pronged radial tree with $k-r$ of its prongs subdivided by adding $n$ vertices of valence $2$ along them. Then $RB_n(R_{k,1}) \cong F_{M_1}$ and $RB_n(R_{k,2}) \cong F_{M_2}$, where 
\begin{align*}
M_1 &= (k-2)\left({n+k-4 \choose k-3} + 2{n+k-4 \choose k-2}\right) - {n+k-2 \choose k-2} + 1\\
M_2 &= (k-2)\left({n+k-3 \choose k-2}+{n+k-5 \choose k-3}\right) - (k-3){n+k-6 \choose k-2} - {n+k-2 \choose k-2} + 1.
\end{align*}
\end{prop}

\begin{proof}
Following the proof of Proposition \ref{prop:radial_trees}, let $e_1, \dots, e_k$ denote the $k$ edges of $R_{k,r}$ sharing the central vertex $v$, ordered so that the edges on subdivided prongs appear first, and apply Theorem \ref{thm:graphs_of_graph_braid_groups} to $e_1, \dots, e_{k-2}$ to obtain a graph of groups decomposition $(\mc{G},\Lambda)$ of $B_n(R_{k,r})$. That is, $e_1, \dots, e_{k-2}$ all lie on subdivided prongs. Note that $R_{k,r} \smallsetminus (\mathring{e}_1 \cup \dots \cup \mathring{e}_{k-2})$ and $R_{k,r} \smallsetminus (e_1 \cup \dots \cup e_{k-2})$ are both disjoint unions of segments, hence the vertex groups and edge groups of $(\mc{G},\Lambda)$ are all trivial. Thus, by Remark \ref{graphs_of_gps_and_spanning_trees}, $RB_n(R_{k,r})$ is a free group $F_{M_r}$ where \[M_r = |E(\Lambda)|-|E(T)| = |E(\Lambda)| - |V(T)| + 1 = |E(\Lambda)| - |V(\Lambda)| + 1\] for some spanning tree $T$ of $\Lambda$.

The computations in the proof of Proposition \ref{prop:radial_trees} then apply, with the caveat that whenever we count partitions, we must bear in mind that the two prongs in $\Gamma_v$ may not have sufficiently many vertices to accommodate all the particles. This affects the values of $|V(\Lambda)|$ and $|E(\Lambda)|$ as detailed below. As in the proof of Proposition \ref{prop:radial_trees}, denote the vertex with $m$ particles in $\Gamma_v$ and $m_i$ particles in $\Gamma_i$, $i = 1, \dots, k-2$, by $K_{m,m_1,\dots,m_{k-2}}$.

If $r=1$, then $\Gamma_v$ still has at least $n$ vertices, so Corollary \ref{cor:no_of_comps_of_UC} may be applied to show that $\Lambda$ has ${n+k-2 \choose k-2}$ vertices, as in Proposition \ref{prop:radial_trees}. Furthermore, by Proposition \ref{prop:counting_edges}, the number of edges connecting the vertices $K_{m,m_1,\dots,m_{k-2}}$ and $K_{m-1,m_1,\dots,m_i+1,\dots,m_{k-2}}$ is equal to the number of ways to partition the $m-1$ particles in $\Gamma_v$ among the two components of $\Gamma_v \smallsetminus v$. Since one component has only one vertex and the other has $n$ vertices, this is equal to $1$ if $m=1$ and $2$ if $m \geq 2$. Thus, 
\begin{align*}
|E(\Lambda)| &= (k-2){n-1+k-3 \choose k-3} + 2(k-2)\sum_{N=0}^{n-2}{N+k-3 \choose k-3}\\
&= (k-2)\left({n+k-4 \choose k-3} + 2{n+k-4 \choose k-2}\right).
\end{align*}
Thus, 
\[ M_1 = |E(\Lambda)| - |V(\Lambda)| + 1 = (k-2)\left({n+k-4 \choose k-3} + 2{n+k-4 \choose k-2}\right) - {n+k-2 \choose k-2} + 1.\]

If $r=2$, then $\Gamma_v$ has $3$ vertices, so $|V(\Lambda)|$ is equal to the number of partitions of $n$ particles into the $k-1$ components of $R_{k,r} \smallsetminus (\mathring{e}_1 \cup \dots \cup \mathring{e}_{k-2})$ with either $0$, $1$, $2$, or $3$ particles in $\Gamma_v$. That is,
\begin{align*}
|V(\Lambda)| &= \sum_{N=0}^{n}{N+k-3 \choose k-3} - \sum_{N=0}^{n-4}{N+k-3 \choose k-3}\\
&= {n+k-2 \choose k-2} - {n+k-6 \choose k-2}.
\end{align*}
For $|E(\Lambda)|$, we are again restricted to the cases $m=0,1,2,3$. Moreover, when $m=3$ there is only one way to partition the $m-1 = 2$ particles among the two components of $\Gamma_v \smallsetminus v$, since each component is a single vertex. Thus, 
\begin{align*}
|E(\Lambda)| &= (k-2)\left({n-3+k-3 \choose k-3}+2{n-2+k-3 \choose k-3}+{n-1+k-3 \choose k-3}\right) \\
&= (k-2)\left(\sum_{N=0}^{n-1}{N+k-3 \choose k-3} - \sum_{N=0}^{n-4}{N+k-3 \choose k-3} + {n+k-5 \choose k-3}\right)\\
&= (k-2)\left({n+k-3 \choose k-2} - {n+k-6 \choose k-2} + {n+k-5 \choose k-3}\right).
\end{align*}
We therefore have
\[ M_2 = (k-2)\left({n+k-3 \choose k-2}+{n+k-5 \choose k-3}\right) - (k-3){n+k-6 \choose k-2} - {n+k-2 \choose k-2} + 1.\]
\end{proof}
 
\begin{prop}\label{prop:H_graph}
Let $\Gamma_H$ be a segment of length $3$ joining two vertices of valence three, as shown in Fig. \ref{fig:example_H}, so that $\Gamma_H$ is homeomorphic to the letter ``H''. Then $B_4(\Gamma_H) \cong F_{10} \ast \mathbb{Z}^2$. 

Furthermore, let $\Gamma'_H$ be the graph obtained from $\Gamma_H$ by adding $4$ vertices of valence $2$ along one edge containing a vertex of valence $1$ on each side of the segment of length $3$. Then $RB_4(\Gamma'_H) \cong F_4 \ast \mathbb{Z}^2$.
\end{prop}

\begin{figure}[ht]
     \centering
     \def\svgscale{0.5}
\begingroup%
  \makeatletter%
  \providecommand\color[2][]{%
    \errmessage{(Inkscape) Color is used for the text in Inkscape, but the package 'color.sty' is not loaded}%
    \renewcommand\color[2][]{}%
  }%
  \providecommand\transparent[1]{%
    \errmessage{(Inkscape) Transparency is used (non-zero) for the text in Inkscape, but the package 'transparent.sty' is not loaded}%
    \renewcommand\transparent[1]{}%
  }%
  \providecommand\rotatebox[2]{#2}%
  \newcommand*\fsize{\dimexpr\f@size pt\relax}%
  \newcommand*\lineheight[1]{\fontsize{\fsize}{#1\fsize}\selectfont}%
  \ifx\svgwidth\undefined%
    \setlength{\unitlength}{198.74503308bp}%
    \ifx\svgscale\undefined%
      \relax%
    \else%
      \setlength{\unitlength}{\unitlength * \real{\svgscale}}%
    \fi%
  \else%
    \setlength{\unitlength}{\svgwidth}%
  \fi%
  \global\let\svgwidth\undefined%
  \global\let\svgscale\undefined%
  \makeatother%
  \begin{picture}(1,0.4339482)%
    \lineheight{1}%
    \setlength\tabcolsep{0pt}%
    \put(0,0){\includegraphics[width=\unitlength,page=1]{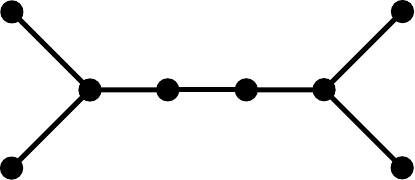}}%
    \put(0.48113159,0.27357928){\color[rgb]{0,0,0}\makebox(0,0)[lt]{\lineheight{1.25}\smash{\begin{tabular}[t]{l}$e$\end{tabular}}}}%
  \end{picture}%
\endgroup%

     \caption{The graph $\Gamma_H$. We wish to split the corresponding graph braid group on $4$ particles along the edge $e$.}
     \label{fig:example_H}
 \end{figure}

\begin{proof}
Let $e$ be the edge in the middle of the segment of length $3$ and subdivide each of the four edges containing vertices of valence $1$ by adding $4$ vertices of valence $2$ along each edge, giving a graph $\Gamma''_H$ homeomorphic to $\Gamma_H$ that satisfies the hypotheses of Theorem \ref{thm:config_space_retraction}. Then $\Gamma''_H \smallsetminus \mathring{e}$ is a disjoint union of two $3$--pronged radial trees $\Gamma''_1$ and $\Gamma''_2$ with two prongs of each tree subdivided $4$ times. That is, $\Gamma''_1$ and $\Gamma''_2$ are copies of $R_{3,1}$. Furthermore, $\Gamma''_H \smallsetminus e$ is a disjoint union of two segments $\Omega''_1$ and $\Omega''_2$ of length $4$. 

Applying Theorem \ref{thm:graphs_of_graph_braid_groups} to $B_4(\Gamma''_H) \cong B_4(\Gamma_H)$, we obtain a graph of groups decomposition $(\mc{G},\Lambda)$, where $\Lambda$ has five vertices $K_{4,0},K_{3,1},K_{2,2},K_{1,3},K_{0,4}$, corresponding to the five partitions of the $4$ particles among the two subgraphs $\Gamma''_1$ and $\Gamma''_2$; that is, $K_{i,j}$ corresponds to the partition with $i$ particles in $\Gamma''_1$ and $j$ particles in $\Gamma''_2$. Applying Proposition \ref{prop:modified_radial_trees} and Lemma \ref{lem:braid_is_product_of_connected_components}, we see that the vertex groups are given by $G_{K_{4,0}} \cong F_{3}$, $G_{K_{3,1}} \cong F_{2}$, $G_{K_{2,2}} \cong \mathbb{Z}^2$, $G_{K_{1,3}} \cong F_{2}$, $G_{K_{0,4}} \cong F_{3}$.

Furthermore, Proposition \ref{prop:determining_adjacency} tells us that $K_{i,j}$ is adjacent to $K_{k,l}$ in $\Lambda$ if and only if $|i-k| = 1$, and since $\Gamma''_H \smallsetminus e$ has the same number of connected components as $\Gamma''_H \smallsetminus \mathring{e}$, by Proposition \ref{prop:counting_edges} there is at most one edge between each pair of vertices of $\Lambda$. Moreover, the edge groups are trivial since $\Omega_1$ and $\Omega_2$ are segments, which have trivial graph braid groups. Thus, we conclude that $B_4(\Gamma_H) \cong F_{3} \ast F_{2} \ast \mathbb{Z}^2 \ast F_{2} \ast F_{3} = F_{10} \ast \mathbb{Z}^2$.

For $RB_4(\Gamma'_H)$, the computation is exactly the same, with the exception that $\Gamma'_H \smallsetminus \mathring{e}$ is a disjoint union of two copies of $R_{3,2}$ rather than $R_{3,1}$. Thus, $RB_4(\Gamma'_H) \cong \mathbb{Z} \ast \mathbb{Z} \ast \mathbb{Z}^2 \ast \mathbb{Z} \ast \mathbb{Z} = F_4 \ast \mathbb{Z}^2$.
\end{proof}

\begin{prop}\label{prop:A_graph}
Let $\Gamma_A$ be a cycle containing two vertices of valence $3$, as shown in Fig. \ref{fig:example_A}, so that $\Gamma_A$ is homeomorphic to the letter ``A''. Then $B_4(\Gamma_A) \cong F_{5} \ast \mathbb{Z}^2$.
\end{prop}

\begin{figure}[ht]
     \centering
     \def\svgscale{0.5}
\begingroup%
  \makeatletter%
  \providecommand\color[2][]{%
    \errmessage{(Inkscape) Color is used for the text in Inkscape, but the package 'color.sty' is not loaded}%
    \renewcommand\color[2][]{}%
  }%
  \providecommand\transparent[1]{%
    \errmessage{(Inkscape) Transparency is used (non-zero) for the text in Inkscape, but the package 'transparent.sty' is not loaded}%
    \renewcommand\transparent[1]{}%
  }%
  \providecommand\rotatebox[2]{#2}%
  \newcommand*\fsize{\dimexpr\f@size pt\relax}%
  \newcommand*\lineheight[1]{\fontsize{\fsize}{#1\fsize}\selectfont}%
  \ifx\svgwidth\undefined%
    \setlength{\unitlength}{198.74503308bp}%
    \ifx\svgscale\undefined%
      \relax%
    \else%
      \setlength{\unitlength}{\unitlength * \real{\svgscale}}%
    \fi%
  \else%
    \setlength{\unitlength}{\svgwidth}%
  \fi%
  \global\let\svgwidth\undefined%
  \global\let\svgscale\undefined%
  \makeatother%
  \begin{picture}(1,0.4339482)%
    \lineheight{1}%
    \setlength\tabcolsep{0pt}%
    \put(0,0){\includegraphics[width=\unitlength,page=1]{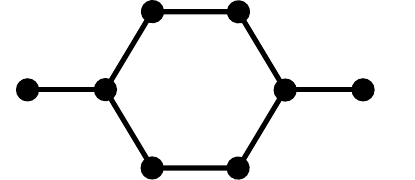}}%
    \put(0.44716848,0.32263686){\color[rgb]{0,0,0}\makebox(0,0)[lt]{\lineheight{1.25}\smash{\begin{tabular}[t]{l}$e$\end{tabular}}}}%
  \end{picture}%
\endgroup%

     \caption{The graph $\Gamma_A$.}
     \label{fig:example_A}
 \end{figure}

\begin{proof}
Let $e$ be the edge indicated in Fig. \ref{fig:example_A} and subdivide each of the two edges containing vertices of valence $1$ by adding $4$ vertices of valence $2$ along each edge, giving a graph $\Gamma'_A$ homeomorphic to $\Gamma_A$ that satisfies the hypotheses of Theorem \ref{thm:config_space_retraction}. Then $\Gamma'_A \smallsetminus \mathring{e} = \Gamma'_H$. Applying Theorem \ref{thm:graphs_of_graph_braid_groups} to $B_4(\Gamma'_A) \cong B_4(\Gamma_A)$, we obtain a graph of groups decomposition $(\mc{G},\Lambda)$, where $\Lambda$ has a single vertex since $\Gamma'_A \smallsetminus \mathring{e} = \Gamma'_H$ is connected, with vertex group $RB_4(\Gamma'_H)$. Furthermore, $\Gamma'_A \smallsetminus e$ is a segment, therefore $\Lambda$ has a single edge by Proposition \ref{prop:counting_edges}, with trivial edge group. Thus, applying Proposition \ref{prop:H_graph} and Remark \ref{graphs_of_gps_and_spanning_trees}, we have $B_4(\Gamma_A) \cong F_{5} \ast \mathbb{Z}^2$.
\end{proof}

\begin{prop}\label{prop:theta_graph}
Let $\Gamma_{\theta}$ be the graph obtained by gluing two $6$--cycles along a segment of length $3$, as shown in Fig. \ref{fig:example_theta}, so that $\Gamma_{\theta}$ is homeomorphic to the letter ``$\theta$''. Then $B_4(\theta)$ is an HNN extension of $\mathbb{Z} \ast \mathbb{Z}^2$.
\end{prop}

\begin{figure}[ht]
     \centering
     \def\svgscale{0.5}
\begingroup%
  \makeatletter%
  \providecommand\color[2][]{%
    \errmessage{(Inkscape) Color is used for the text in Inkscape, but the package 'color.sty' is not loaded}%
    \renewcommand\color[2][]{}%
  }%
  \providecommand\transparent[1]{%
    \errmessage{(Inkscape) Transparency is used (non-zero) for the text in Inkscape, but the package 'transparent.sty' is not loaded}%
    \renewcommand\transparent[1]{}%
  }%
  \providecommand\rotatebox[2]{#2}%
  \newcommand*\fsize{\dimexpr\f@size pt\relax}%
  \newcommand*\lineheight[1]{\fontsize{\fsize}{#1\fsize}\selectfont}%
  \ifx\svgwidth\undefined%
    \setlength{\unitlength}{97.3655977bp}%
    \ifx\svgscale\undefined%
      \relax%
    \else%
      \setlength{\unitlength}{\unitlength * \real{\svgscale}}%
    \fi%
  \else%
    \setlength{\unitlength}{\svgwidth}%
  \fi%
  \global\let\svgwidth\undefined%
  \global\let\svgscale\undefined%
  \makeatother%
  \begin{picture}(1,0.98238373)%
    \lineheight{1}%
    \setlength\tabcolsep{0pt}%
    \put(0,0){\includegraphics[width=\unitlength,page=1]{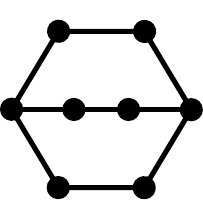}}%
    \put(0.45378226,0.88754132){\color[rgb]{0,0,0}\makebox(0,0)[lt]{\lineheight{1.25}\smash{\begin{tabular}[t]{l}$e$\end{tabular}}}}%
  \end{picture}%
\endgroup%

     \caption{The graph $\Gamma_\theta$.}
     \label{fig:example_theta}
 \end{figure}

\begin{proof}
Let $e$ be the edge indicated in Fig. \ref{fig:example_theta} and note that $\Gamma_\theta \smallsetminus \mathring{e} = \Gamma_A$ and $\Gamma_\theta \smallsetminus e$ is a $6$--cycle. Applying Theorem \ref{thm:graphs_of_graph_braid_groups} and Proposition \ref{prop:counting_edges} to $B_4(\Gamma_\theta)$, we obtain a graph of groups decomposition $(\mc{G},\Lambda)$, where $\Lambda$ has a single vertex with vertex group $RB_4(\Gamma_A)$ and a single edge with edge group $\mathbb{Z}$. Thus, $B_4(\Gamma_\theta)$ is an HNN extension of $RB_4(\Gamma_A)$.

Furthermore, we may apply Theorem \ref{thm:graphs_of_graph_braid_groups} and Proposition \ref{prop:counting_edges} to $RB_4(\Gamma_A)$ using the edge indicated in Fig. \ref{fig:example_A}, noting that $\Gamma_A \smallsetminus \mathring{e} = \Gamma_H$ and $\Gamma_A \smallsetminus e$ is a segment. Thus, we obtain a graph of groups decomposition $(\mc{G}',\Lambda')$, where $\Lambda'$ has a single vertex with vertex group $RB_4(\Gamma_H)$ and a single edge with trivial edge group. Thus, $RB_4(\Gamma_A) \cong RB_4(\Gamma_H) \ast \mathbb{Z}$.

Finally, apply Theorem \ref{thm:graphs_of_graph_braid_groups} and Propositions \ref{prop:determining_adjacency} and \ref{prop:counting_edges} to $RB_4(\Gamma_H)$ using the edge indicated in Fig. \ref{fig:example_H}, noting that $\Gamma_H \smallsetminus \mathring{e}$ is a disjoint union of two copies of the $3$--pronged radial tree $R_3$ and $\Gamma_H \smallsetminus e$ is a disjoint union of two segments. Thus, we obtain a graph of groups decomposition $(\mc{G}'',\Lambda'')$, where $\Lambda''$ has five vertices $K_{4,0},K_{3,1},K_{2,2},K_{1,3},K_{0,4}$ corresponding to the five partitions of the $4$ particles among the two copies of $R_3$, with $K_{i,j}$ adjacent to $K_{k,l}$ in $\Lambda''$ if and only if $|i-k|=1$, in which case they are connected by a single edge. 

Moreover, the vertex group associated to $K_{i,j}$ is $RB_i(R_3) \times RB_j(R_3)$ by Lemma \ref{lem:braid_is_product_of_connected_components}, and the edge groups are all trivial. Note that: $UC_0(R_3)$ is a single point hence $RB_0(R_3)$ is trivial; $UC_1(R_3) = R_3$ hence $RB_1(R_3)$ is trivial; $RB_2(R_3) \cong B_2(R_3) \cong \mathbb{Z}$ by Proposition \ref{prop:radial_trees}; and $UC_4(R_3)$ is a single point since $R_3$ has only $4$ vertices, so $RB_4(R_3)$ is trivial. Furthermore, since $R_3$ has $4$ vertices, we have $UC_3(R_3) = UC_1(R_3)$. This can be seen by considering configurations of particles on vertices of $R_3$ as colorings of the vertices of $R_3$ in black or white according to whether they are occupied by a particle. By inverting the coloring, we see that $UC_3(R_3) = UC_1(R_3)$. Thus, $RB_3(R_3)$ is trivial. It therefore follows that $RB_4(\Gamma_H) \cong \mathbb{Z}^2$.

We therefore conclude that $B_4(\Gamma_\theta)$ is an HNN extension of $RB_4(\Gamma_A) \cong RB_4(\Gamma_H) \ast \mathbb{Z} \cong \mathbb{Z}^2 \ast \mathbb{Z}$.
\end{proof}

Note that Propositions \ref{prop:H_graph}, \ref{prop:A_graph}, and \ref{prop:theta_graph} answer \cite[Question 5.6]{Gen_braids}, up to determining whether $B_4(\Gamma_\theta)$ is a non-trivial free product. This is summarized in the following theorem.

\begin{thm}\label{thm:toral_rel_hyp_and_free_prods}
Let $\Gamma$ be a finite connected graph. The braid group $B_4(\Gamma)$ is toral relatively hyperbolic only if it is either a free group or isomorphic to $F_{10} \ast \mathbb{Z}^2$ or $F_{5} \ast \mathbb{Z}^2$ or an HNN extension of $\mathbb{Z} \ast \mathbb{Z}^2$.
\end{thm}

\begin{proof}
By \cite[Theorem 4.22]{Gen_braids}, $B_4(\Gamma)$ is toral relatively hyperbolic if and only if $\Gamma$ is either: a collection of cycles and segments glued along a single central vertex (called a \emph{flower graph}); a graph $\Gamma_H$ homeomorphic to the letter ``H''; a graph $\Gamma_A$ homeomorphic to the letter ``A''; or a graph $\Gamma_\theta$ homeomorphic to the letter ``$\theta$''. Braid groups of flower graphs are free by \cite[Corollary 4.7]{Gen_braids}, while $B_4(\Gamma_H)$ and $B_4(\Gamma_A)$ are isomorphic to $F_{10} \ast \mathbb{Z}^2$ and $F_{5} \ast \mathbb{Z}^2$ respectively, by Propositions \ref{prop:H_graph} and \ref{prop:A_graph}, and $B_4(\Gamma_\theta)$ is an HNN extension of $\mathbb{Z} \ast \mathbb{Z}^2$ by Proposition \ref{prop:theta_graph}.
\end{proof}

We now compute the braid group of the graph homeomorphic to the letter ``Q'', for any number of particles.

\begin{prop}\label{prop:Q_graph}
Let $\Gamma_Q$ be a cycle containing one vertex of valence $3$, as shown in Fig. \ref{fig:example_Q}, so that $\Gamma_Q$ is homeomorphic to the letter ``Q''. Then $B_n(\Gamma_Q) \cong F_n$.
\end{prop}

\begin{figure}[ht]
     \centering
     \def\svgscale{0.5}
\begingroup%
  \makeatletter%
  \providecommand\color[2][]{%
    \errmessage{(Inkscape) Color is used for the text in Inkscape, but the package 'color.sty' is not loaded}%
    \renewcommand\color[2][]{}%
  }%
  \providecommand\transparent[1]{%
    \errmessage{(Inkscape) Transparency is used (non-zero) for the text in Inkscape, but the package 'transparent.sty' is not loaded}%
    \renewcommand\transparent[1]{}%
  }%
  \providecommand\rotatebox[2]{#2}%
  \newcommand*\fsize{\dimexpr\f@size pt\relax}%
  \newcommand*\lineheight[1]{\fontsize{\fsize}{#1\fsize}\selectfont}%
  \ifx\svgwidth\undefined%
    \setlength{\unitlength}{134.76750291bp}%
    \ifx\svgscale\undefined%
      \relax%
    \else%
      \setlength{\unitlength}{\unitlength * \real{\svgscale}}%
    \fi%
  \else%
    \setlength{\unitlength}{\svgwidth}%
  \fi%
  \global\let\svgwidth\undefined%
  \global\let\svgscale\undefined%
  \makeatother%
  \begin{picture}(1,0.63967288)%
    \lineheight{1}%
    \setlength\tabcolsep{0pt}%
    \put(0,0){\includegraphics[width=\unitlength,page=1]{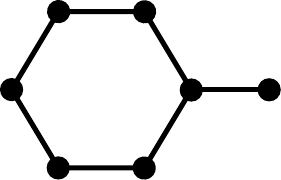}}%
    \put(0.6283621,0.48122502){\color[rgb]{0,0,0}\makebox(0,0)[lt]{\lineheight{1.25}\smash{\begin{tabular}[t]{l}$e$\end{tabular}}}}%
  \end{picture}%
\endgroup%

     \caption{The graph $\Gamma_Q$.}
     \label{fig:example_Q}
 \end{figure}

\begin{proof}
Let $e$ be the edge indicated in Fig. \ref{fig:example_Q}, and subdivide all other edges of $\Gamma_Q$ by adding $n$ vertices of valence $2$ along them, so that the hypotheses of Theorem \ref{thm:config_space_retraction} are satisfied. Note that $\Gamma_Q \smallsetminus \mathring{e}$ is a segment, while $\Gamma_Q \smallsetminus e$ is a disjoint union of two segments. Applying Theorem \ref{thm:graphs_of_graph_braid_groups} and Proposition \ref{prop:counting_edges} to $B_n(\Gamma_Q)$, we obtain a graph of groups decomposition $(\mc{G},\Lambda)$, where $\Lambda$ has a single vertex with trivial vertex group and ${n \choose 1}=n$ edges with trivial edge groups. Thus, Remark \ref{graphs_of_gps_and_spanning_trees} tells us that $B_n(\Gamma_Q) \cong F_n$.
\end{proof}

Next, we compute braid groups of wheel graphs.

\begin{prop}\label{prop:wheels}
Let $W_m$ be the wheel graph with $m$ spokes; that is, $W_m$ consists of a single central vertex connected by an edge to all vertices of an $m$--cycle (see Fig. \ref{fig:example_W}). Then $B_2(W_m) \cong F_{m+1}$.
\end{prop}

\begin{figure}[ht]
     \centering
     \def\svgscale{0.6}
\begingroup%
  \makeatletter%
  \providecommand\color[2][]{%
    \errmessage{(Inkscape) Color is used for the text in Inkscape, but the package 'color.sty' is not loaded}%
    \renewcommand\color[2][]{}%
  }%
  \providecommand\transparent[1]{%
    \errmessage{(Inkscape) Transparency is used (non-zero) for the text in Inkscape, but the package 'transparent.sty' is not loaded}%
    \renewcommand\transparent[1]{}%
  }%
  \providecommand\rotatebox[2]{#2}%
  \newcommand*\fsize{\dimexpr\f@size pt\relax}%
  \newcommand*\lineheight[1]{\fontsize{\fsize}{#1\fsize}\selectfont}%
  \ifx\svgwidth\undefined%
    \setlength{\unitlength}{97.3655977bp}%
    \ifx\svgscale\undefined%
      \relax%
    \else%
      \setlength{\unitlength}{\unitlength * \real{\svgscale}}%
    \fi%
  \else%
    \setlength{\unitlength}{\svgwidth}%
  \fi%
  \global\let\svgwidth\undefined%
  \global\let\svgscale\undefined%
  \makeatother%
  \begin{picture}(1,0.88539606)%
    \lineheight{1}%
    \setlength\tabcolsep{0pt}%
    \put(0,0){\includegraphics[width=\unitlength,page=1]{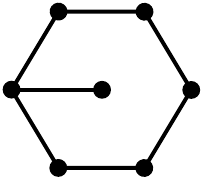}}%
    \put(0.67524148,0.5004693){\color[rgb]{0,0,0}\makebox(0,0)[lt]{\lineheight{1.25}\smash{\begin{tabular}[t]{l}$e_1$\end{tabular}}}}%
    \put(0,0){\includegraphics[width=\unitlength,page=2]{example_W6.pdf}}%
    \put(0.44415367,0.67185972){\color[rgb]{0,0,0}\makebox(0,0)[lt]{\lineheight{1.25}\smash{\begin{tabular}[t]{l}$e_2$\end{tabular}}}}%
    \put(0.24580335,0.57364704){\color[rgb]{0,0,0}\makebox(0,0)[lt]{\lineheight{1.25}\smash{\begin{tabular}[t]{l}$e_3$\end{tabular}}}}%
    \put(0.21114018,0.33485643){\color[rgb]{0,0,0}\makebox(0,0)[lt]{\lineheight{1.25}\smash{\begin{tabular}[t]{l}$e_4$\end{tabular}}}}%
    \put(0.40371333,0.18657508){\color[rgb]{0,0,0}\makebox(0,0)[lt]{\lineheight{1.25}\smash{\begin{tabular}[t]{l}$e_5$\end{tabular}}}}%
  \end{picture}%
\endgroup%

     \caption{The graph $W_6$.}
     \label{fig:example_W}
 \end{figure}

\begin{proof}
Let $e_1, \dots, e_{m-1}$ be distinct spokes of $W_m$, as illustrated in Fig. \ref{fig:example_W}. Note that $W_m \smallsetminus (\mathring{e}_1 \cup \dots \cup \mathring{e}_{m-1})$ is homeomorphic to $\Gamma_Q$, while $W_m \smallsetminus e_i$ is a segment for each $i$. Applying Theorem \ref{thm:graphs_of_graph_braid_groups} to $B_2(W_m)$, we obtain a graph of groups decomposition $(\mc{G},\Lambda)$, where $\Lambda$ has a single vertex with vertex group $B_2(\Gamma_Q)$ and $m-1$ edges with trivial edge groups. It therefore follows from Proposition \ref{prop:Q_graph} that $B_2(W_m) \cong F_{m+1}$.
\end{proof}

We conclude our examples by discussing the cases of complete graphs and complete bipartite graphs.

\begin{ex}\label{ex:complete}
Note that $K_4 = W_3$, so Proposition \ref{prop:wheels} shows that $B_2(K_4) \cong F_4$, answering a question of Genevois \cite[Example 4.10]{Gen_braids}. We may also compute graph of groups decompositions of $B_2(K_m)$ for other values of $m$ by applying Theorem \ref{thm:graphs_of_graph_braid_groups} with edges $e_1, \dots, e_{m-1}$ sharing a common vertex $v$ of $K_m$. In this case, $K_m \smallsetminus (\mathring{e}_1 \cup \dots \cup \mathring{e}_{m-1}) = v \sqcup K_{m-1}$ and $K_m \smallsetminus e_i = K_{m-2}$ for each $i$. Thus, $B_2(K_m)$ decomposes as a graph of groups $(\mc{G},\Lambda)$, where $\Lambda$ consists of two vertices joined by $m-1$ edges, the vertex groups are $B_2(K_{m-1})$ and $B_1(K_{m-1}) \cong F_{(m-2)(m-3)/2}$, and the edge groups are all $B_1(K_{m-2}) \cong F_{(m-3)(m-4)/2}$. We may therefore compute $B_2(K_m)$ inductively, with the base case $B_2(K_4) \cong F_4$.

Note that while all the vertex groups and edge groups in this construction are free, these braid groups are still quite difficult to analyze in general, often having unexpected properties. For example, it follows from \cite[Example 5.1]{AbramsThesis} that $B_2(K_5)$ is the fundamental group of a closed non-orientable surface of genus six, which is not readily apparent from the graph of groups.
\end{ex}

\begin{ex}\label{ex:complete_bipartite}
We may compute $B_2(K_{p,q})$ in a similar way to $B_2(K_m)$. Suppose $p \geq q$ and apply Theorem \ref{thm:graphs_of_graph_braid_groups} to edges $e_1, \dots, e_q$ of $K_{p,q}$ sharing a common vertex of valence $q$. This produces a graph of groups decomposition $(\mc{G},\Lambda)$ of $B_2(K_{p,q})$, where $\Lambda$ consists of two vertices joined by $q$ edges, the vertex groups are $B_2(K_{p-1,q})$ and $B_1(K_{p-1,q}) \cong F_{pq-p-2q+2}$, and the edge groups are all $B_1(K_{p-1,q-1}) \cong F_{pq-2p-2q+4}$. We may therefore compute $B_2(K_{p,q})$ inductively, with the base case $B_2(K_{2,2}) \cong \mathbb{Z}$. 

Once again, despite the fact that $B_2(K_{p,q})$ is constructed from free groups, it is difficult to extract information about the structure of the group as a whole. For example, it follows from \cite[Example 5.2]{AbramsThesis} that $B_2(K_{3,3})$ is the fundamental group of a non-orientable surface of genus four, but this is not easy to see from the graph of groups.
\end{ex}

\section{Applications}\label{sec:applications}

\subsection{\textit{Decomposing graph braid groups as free products}}\label{sec:free_prods}
In Section \ref{sec:examples}, we saw that it appears to be quite common for a graph braid group to decompose as a non-trivial free product. We now prove two general criteria for a graph braid group to decompose as a non-trivial free product, and apply these to answer a question of Genevois \cite[Question 5.3]{Gen_braids}. 

We define the following classes of graphs (c.f. \cite[Section 4.1]{Gen_braids}). See Fig. \ref{fig:flower_sun_pulsar_theta} for examples.
\begin{itemize}
\item A \emph{flower graph} is a graph obtained by gluing cycles and segments along a single central vertex.
\item A \emph{sun graph} is a graph obtained by gluing segments to vertices of a cycle. 
\item A \emph{pulsar graph} is obtained by gluing cycles along a fixed non-trivial segment $\Omega$ and gluing further segments to the endpoints of $\Omega$. 
\item The \emph{generalized theta graph} $\Theta_m$ is the pulsar graph obtained by gluing $m$ cycles along a fixed segment.
\end{itemize}

\begin{figure}[ht]
     \centering
     \def\svgscale{0.5}
     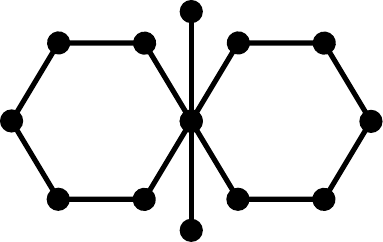\hspace{0.5in}\def\svgscale{0.5}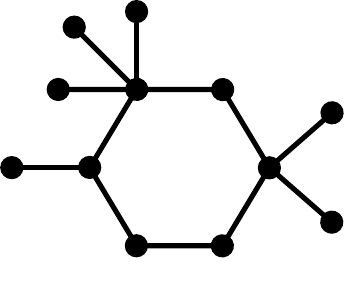\hspace{0.5in}\def\svgscale{0.5}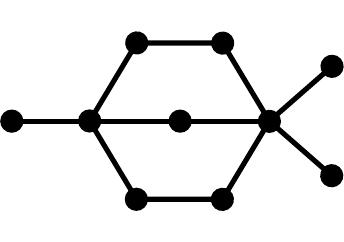\hspace{0.3in}\def\svgscale{0.5}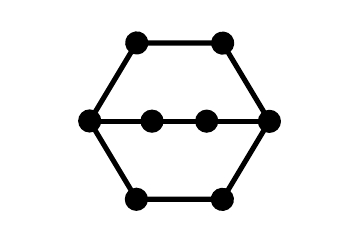
     \caption{Left to right: a flower graph, a sun graph, a pulsar graph, and $\Theta_2$.}
     \label{fig:flower_sun_pulsar_theta}
 \end{figure}

\begin{lem}[Free product criterion 1]\label{lem:gbg_free_product_1}
Let $n \geq 2$ and let $\Gamma$ be a finite graph obtained by gluing a non-segment flower graph $\Phi$ to a connected non-segment graph $\Omega$ along a vertex $v$, where $v$ is either the central vertex of $\Phi$ or a vertex of $\Phi$ of valence $1$. Then $B_n(\Gamma) \cong H \ast \mathbb{Z}$ for some non-trivial subgroup $H$ of $B_n(\Gamma)$.
\end{lem}

\begin{proof}
Consider $\Gamma$ as the union $\Gamma = \Phi \cup \Omega$, where $\Phi \cap \Omega = v$ is either the central vertex $v_c$ of $\Phi$ or a vertex of $\Phi$ of valence $1$. If $v \neq v_c$ and $\Phi$ is not a $3$--pronged radial tree, then we may remove the segment connecting $v$ to $v_c$ from $\Phi$ and add it to $\Omega$. The new $\Phi$ and $\Omega$ will still satisfy the hypotheses of the lemma, and $\Phi \cap \Omega = v_c$. If $\Phi$ is a $3$--pronged radial tree, then after subdividing edges sufficiently we may remove all but the last edge of the segment connecting $v$ to $v_c$ from $\Phi$ and add it to $\Omega$, so that $\Phi \cap \Omega = v$ where $v$ is a vertex of valence $2$ in $\Gamma$ adjacent to $v_c$. Thus, henceforth we shall assume without loss of generality that $\Phi \cap \Omega = v$ where either $v = v_c$ or $\Phi$ is a $3$--pronged radial tree and $v$ is a vertex of valence $2$ in $\Gamma$ adjacent to $v_c$.

Let $e_1, \dots, e_m$ denote the edges of $\Omega$ that have $v$ as an endpoint and apply Theorem \ref{thm:graphs_of_graph_braid_groups} to obtain a graph of groups decomposition $(\mc{G},\Lambda)$ of $B_n(\Gamma)$. Note that $\Gamma \smallsetminus (\mathring{e}_1 \cup \dots \cup \mathring{e}_m)$ is a disjoint union of $\Phi$ and some subgraphs $\Omega_1, \dots, \Omega_k$ of $\Omega$. By sufficiently subdividing the edges of $\Gamma$, we may assume without loss of generality that the graphs $\Gamma,\Phi,\Omega_1,\dots,\Omega_k$ all satisfy the hypotheses of Theorem \ref{thm:config_space_retraction}. The vertices of $\Lambda$ correspond to the partitions of the $n$ particles among $\Phi, \Omega_1, \dots, \Omega_k$. Let $K_{p,p_1,\dots,p_k}$ denote the partition with $p$ particles in $\Phi$ and $p_i$ particles in $\Omega_i$, $i = 1, \dots, k$. Proposition \ref{prop:determining_adjacency} tells us that $K := K_{n,0,\dots,0}$ is adjacent to precisely the vertices $K^1_i := K_{n-1,0,\dots,1,\dots,0}$ for $i = 1, \dots,k$, where $K^1_i$ is the partition with $n-1$ particles in $\Phi$ and one particle in $\Omega_i$. Furthermore, $K^1_i$ is adjacent to precisely $K$ and the vertices $K_{i,j}^{1,1}$ for $j \neq i$ and $K^2_i$, where $K_{i,j}^{1,1}$ is the partition with $n-2$ particles in $\Phi$, one particle in $\Omega_i$ and one particle in $\Omega_j$, and $K^2_i$ is the partition with $n-2$ particles in $\Phi$ and two particles in $\Omega_i$.

First suppose $\Omega_i$ contains a cycle for some $i$. We have $G_K \cong B_n(\Phi)$, which is a free group by \cite[Corollary 4.7]{Gen_braids} and is non-trivial by Lemma \ref{lem:infinite_diameter_braid_criterion} since $n \geq 2$ and $\Phi$ is not a segment. We also have $G_{K_i^1} \cong B_{n-1}(\Phi) \times B_1(\Omega_i)$ by Lemma \ref{lem:braid_is_product_of_connected_components}, which is non-trivial for some $i$ by Lemma \ref{lem:infinite_diameter_braid_criterion} since some $\Omega_i$ contains a cycle. Furthermore, for each $i$ the edges connecting $K$ to $K_i^1$ all have edge groups isomorphic to $RB_{n-1}(\Phi \smallsetminus v)$, which is trivial since $\Phi \smallsetminus v$ is a disjoint union of segments. Thus, by Remark \ref{graphs_of_gps_and_spanning_trees} $B_n(\Gamma)$ decomposes as a free product of the form $B_n(\Gamma) \cong H \ast B_n(\Phi) \cong H \ast F \cong H' \ast \mathbb{Z}$ for some non-trivial free group $F$ and non-trivial group $H$.

Now suppose that $\Omega_i$ does not contain a cycle for any $i$. If $\Omega_i$ are all segments, then $\Omega$ is a non-segment flower graph. Furthermore, $v$ is either the central vertex of $\Omega$ (in which case $v$ has valence $\geq 3$ in $\Gamma$ and so we must have $v=v_c$) or $\Omega$ is a $3$--pronged radial tree and $v$ is a vertex of $\Omega$ of valence $1$. In the former case, $\Gamma$ is a flower graph containing a vertex of valence $\geq 4$, thus $B_n(\Gamma)$ is a free group by \cite[Corollary 4.7]{Gen_braids}, and by Lemma \ref{lem:subgraphs_induce_subgroups} and Proposition \ref{prop:radial_trees} its rank is at least $2$.  That is, $B_n(\Gamma) \cong H \ast \mathbb{Z}$ for some non-trivial subgroup $H$, as required. In the latter case, by subdividing edges of $\Omega$ we may assume $\Omega \smallsetminus v$ is also a $3$--pronged radial tree. Thus, we may assume some $\Omega_i$ contains a vertex of valence $\geq 3$. 

Note that for each $i$, the edges connecting $K_i^1$ to $K_{i,j}^{1,1}$ and the edges connecting $K_i^1$ to $K_i^2$ all have edge groups isomorphic to $RB_{n-2}(\Phi\smallsetminus v) \times RB_1(\Omega_i \smallsetminus v)$, which is trivial since $\Phi \smallsetminus v$ is a disjoint union of segments and $\Omega_i \smallsetminus v$ does not contain a cycle. Furthermore, we have $G_{K_i^2} \cong B_{n-2}(\Phi) \times B_2(\Omega_i)$, which is non-trivial for some $i$ by Lemma \ref{lem:infinite_diameter_braid_criterion} since some $\Omega_i$ contains a vertex of valence $\geq 3$. Thus, by Remark \ref{graphs_of_gps_and_spanning_trees} we again have $B_n(\Gamma) \cong H \ast \mathbb{Z}$, for some non-trivial group $H$.
\end{proof}

\begin{lem}[Free product criterion 2]\label{lem:gbg_free_product_2}
Let $n \geq 2$ and let $\Gamma$ be a finite graph satisfying the hypotheses of Theorem \ref{thm:config_space_retraction}. Suppose $\Gamma$ contains an edge $e$ such that $\Gamma \smallsetminus \mathring{e}$ is connected but $\Gamma \smallsetminus e$ is disconnected, and one of the connected components of $\Gamma \smallsetminus e$ is a segment of length at least $n-2$. Then $B_n(\Gamma) \cong H \ast \mathbb{Z}$ for some non-trivial subgroup $H$ of $B_n(\Gamma)$.
\end{lem}

\begin{proof}
First suppose $\Gamma \smallsetminus \mathring{e}$ is a segment. Then $\Gamma$ must consist of a single cycle, containing $e$, with a segment glued to either one or both endpoints of $e$ (there must be at least one such segment since $\Gamma \smallsetminus e$ is disconnected). That is, $\Gamma$ is homeomorphic to either $\Gamma_Q$ or $\Gamma_A$ -- the graphs homeomorphic to the letters ``Q'' and ``A'', respectively. 

If $\Gamma$ is homeomorphic to $\Gamma_Q$, then Proposition \ref{prop:Q_graph} tells us $B_n(\Gamma) \cong F_n$. Since $n \geq 2$, this implies $B_n(\Gamma) \cong H \ast \mathbb{Z}$ for some non-trivial subgroup $H$. If $\Gamma$ is homeomorphic to $\Gamma_A$, then we may subdivide the edge $e$ to obtain a graph $\Gamma'$ homeomorphic to $\Gamma$ and containing no edge $e'$ such that $\Gamma' \smallsetminus \mathring{e}'$ is a segment. Furthermore, $\Gamma'$ still contains an edge $e''$ such that $\Gamma' \smallsetminus \mathring{e}''$ is connected but $\Gamma' \smallsetminus e''$ is disconnected, and one of the connected components is a segment. Thus, we may assume without loss of generality that we are in the case where $\Gamma \smallsetminus \mathring{e}$ is not a segment.

Suppose $\Gamma \smallsetminus \mathring{e}$ is not a segment. Applying Theorem \ref{thm:graphs_of_graph_braid_groups}, we obtain a graph of groups decomposition $(\mc{G},\Lambda)$ of $B_n(\Gamma)$ where $\Lambda$ has a single vertex with vertex group $RB_n(\Gamma \smallsetminus \mathring{e})$, which is non-trivial since $\Gamma \smallsetminus \mathring{e}$ is not a segment. Furthermore, $\Gamma \smallsetminus e$ is disconnected with a segment $\Omega$ as one of its connected components. Let $S \in UC_n(\Gamma \smallsetminus \mathring{e})^{(0)}$ be a configuration with one particle at an endpoint of $e$ and all remaining $n-1$ particles in the segment $\Omega$; this configuration exists since $\Omega$ has length at least $n-2$. Then $\Lambda$ has an edge with edge group $RB_{n-1}(\Gamma \smallsetminus e, S \cap (\Gamma \smallsetminus e))$, which is isomorphic to $RB_{n-1}(\Omega) = 1$ by Lemma \ref{lem:braid_is_product_of_connected_components}. Thus, Remark \ref{graphs_of_gps_and_spanning_trees} tells us $B_n(\Gamma)$ decomposes as a free product $B_n(\Gamma) \cong H \ast \mathbb{Z}$, where $H$ is non-trivial since the vertex group $RB_n(\Gamma \smallsetminus \mathring{e})$ is non-trivial.
\end{proof}

Note that in particular, Lemma \ref{lem:gbg_free_product_2} implies that all braid groups on sun graphs decompose as non-trivial free products, with the exception of the cycle graph, which has graph braid group isomorphic to $\mathbb{Z}$. Lemma \ref{lem:gbg_free_product_2} also implies that all braid groups on pulsar graphs decompose as non-trivial free products, with the possible exception of generalized theta graphs $\Theta_m$ for $m \geq 2$. Thus, we have partially answered \cite[Question 5.3]{Gen_braids}.

\begin{thm}\label{thm:hyperbolicity_and_free_products}
Let $\Gamma$ be a finite connected graph that is not homeomorphic to $\Theta_m$ for any $m \geq 0$. The braid group $B_3(\Gamma)$ is hyperbolic only if $B_3(\Gamma) \cong H \ast \mathbb{Z}$ for some group $H$.
\end{thm}

\begin{proof}
By \cite[Theorem 4.1]{Gen_braids}, $B_3(\Gamma)$ is hyperbolic if and only if $\Gamma$ is a tree, a flower graph, a sun graph, or a pulsar graph. If $\Gamma$ is a tree, then either $\Gamma$ is a radial tree, in which case $B_3(\Gamma)$ is free of rank $M \geq 3$ by Proposition \ref{prop:radial_trees}, or $\Gamma$ has at least two vertices of valence $\geq 3$, in which case $B_3(\Gamma) \cong H \ast \mathbb{Z}$ by Lemma \ref{lem:gbg_free_product_1}. In fact, $B_3(\Gamma)$ is known to be free by \cite[Theorems 2.5, 4.3]{FarleySabalka_DMT}. If $\Gamma$ is a flower graph, then $B_3(\Gamma)$ is free by \cite[Corollary 4.7]{Gen_braids}. Suppose $\Gamma$ is a sun graph or a pulsar graph that contains at least one cycle and is not homeomorphic to $\Theta_m$ for any $m$. Subdivide edges of $\Gamma$ so that it satisfies the hypotheses of Theorem \ref{thm:config_space_retraction}. Then $\Gamma$ has an edge $e$ that is contained in a cycle and incident to a vertex of degree at least $3$ with a segment of length at least $2$ glued to it. Thus, $\Gamma \smallsetminus \mathring{e}$ is connected but $\Gamma \smallsetminus e$ is disconnected, and one of the connected components of $\Gamma \smallsetminus e$ is a segment of length at least $1$. Thus, $B_3(\Gamma) \cong H \ast \mathbb{Z}$ for some group $H$ by Lemma \ref{lem:gbg_free_product_2}. If $\Gamma$ is a pulsar graph that contains no cycles and is not homeomorphic to $\Theta_0$, then $\Gamma$ consists of a segment with further segments glued to the endpoints. Thus, we may apply Lemma \ref{lem:gbg_free_product_1} to conclude that $B_3(\Gamma) \cong H \ast \mathbb{Z}$ for some subgroup $H$. 
\end{proof}

\subsection{\textit{Pathological relative hyperbolicity in graph braid groups}}\label{sec:rel_hyp}

We can apply Theorem \ref{thm:graphs_of_graph_braid_groups} to subgraphs of the graph $\Delta'$ given in Fig. \ref{fig:example_graphs_and_cube_cxs_2} to provide a negative answer to a question of Genevois regarding relative hyperbolicity of graph braid groups \cite[follow-up to Question 5.7]{Gen_braids}.

\begin{thm}\label{thm:counterexample}
Let $\Delta'$ be the graph given in Fig. \ref{fig:example_graphs_and_cube_cxs_2}. The graph braid group $B_{2}(\Delta')$ is hyperbolic relative to a thick, proper subgroup $P$ that is not contained in any graph braid group of the form $B_k(\Lambda,S)$ for $k \leq 2$ and $\Lambda \subsetneq \Delta'$. 
\end{thm}

\begin{proof}
Recall that in Example \ref{ex:HNN}, we showed that 
\begin{align*}
B_2(\Delta') = \pi_1(UC_2(\Delta')) &\cong \sub{a,b,c,d,e,f,g\,|\,aba^{-1}b^{-1},\,bcb^{-1}c^{-1},\,cdc^{-1}d^{-1}}\\
&= \pi_1(T_1) \ast_{\sub{b}} \pi_1(T_2) \ast_{\sub{c}} \pi_1(T_3) \ast F_3,
\end{align*}
where $\pi_{1}(T_1) = \sub{a,b}$, $\pi_{1}(T_2) = \sub{b,c}$, $\pi_{1}(T_3) = \sub{c,d}$, and $b$ and $c$ correspond to the simple closed curves $\beta$ and $\gamma$.

The group $B_2(\Delta')$ is therefore hyperbolic relative to 
\[P = \pi_1(T_1) \ast_{\sub{b}} \pi_1(T_2) \ast_{\sub{c}} \pi_1(T_3).\]
Note that $P$ is the group constructed by Croke and Kleiner in \cite[Section 3]{Croke_Kleiner_gp}. In particular, it is isomorphic to the right-angled Artin group $A_\Pi$ where $\Pi$ is a segment of length $3$; this is thick of order $1$ by \cite[Corollary 10.8]{behrstock_Drutu_Mosher_Thickness} and \cite[Corollary 5.4]{BehrstockCharney_divergence}. 

Suppose $P$ is contained in a graph braid group of the form $B_k(\Lambda,S)$ for $k \leq 2$ and $\Lambda \subsetneq \Delta'$.  Since $P$ contains a $\mathbb{Z}^2$ subgroup, we may assume by Theorem \ref{thm:Z2_characterisation} that $k=2$ and $\Lambda$ contains two disjoint cycles. If $S$ has both particles in a single connected component $\Lambda'$ of $\Lambda$, then we have $B_2(\Lambda,S) \cong B_2(\Lambda')$, so in this case we may assume $\Lambda$ is connected. On the other hand, if $S$ has one particle in one connected component $\Lambda'$ and the other particle in another component $\Lambda''$, then $B_2(\Lambda,S)$ is a direct product of two free groups by Lemma \ref{lem:braid_is_product_of_connected_components}. It follows that $P \cong A_\Pi$ must virtually be a direct product of two free groups by \cite[Theorem 2]{BaumslagRoseblade}. However, this is impossible since the class of right-angled Artin groups whose defining graphs are trees of diameter greater than $2$ is quasi-isometrically rigid \cite[Theorem 5.3]{BN_graph_mfd}. Thus, $\Lambda$ must be connected and we may therefore deduce that $\Lambda$ must be isomorphic to one of the six graphs $\Lambda_i$, $i=1,\dots,6$, given in Fig. \ref{fig:counterexample_Lambdas}.  

\begin{figure}[ht]
     \centering
     \def\svgscale{0.45}
\begingroup%
  \makeatletter%
  \providecommand\color[2][]{%
    \errmessage{(Inkscape) Color is used for the text in Inkscape, but the package 'color.sty' is not loaded}%
    \renewcommand\color[2][]{}%
  }%
  \providecommand\transparent[1]{%
    \errmessage{(Inkscape) Transparency is used (non-zero) for the text in Inkscape, but the package 'transparent.sty' is not loaded}%
    \renewcommand\transparent[1]{}%
  }%
  \providecommand\rotatebox[2]{#2}%
  \newcommand*\fsize{\dimexpr\f@size pt\relax}%
  \newcommand*\lineheight[1]{\fontsize{\fsize}{#1\fsize}\selectfont}%
  \ifx\svgwidth\undefined%
    \setlength{\unitlength}{170.38004237bp}%
    \ifx\svgscale\undefined%
      \relax%
    \else%
      \setlength{\unitlength}{\unitlength * \real{\svgscale}}%
    \fi%
  \else%
    \setlength{\unitlength}{\svgwidth}%
  \fi%
  \global\let\svgwidth\undefined%
  \global\let\svgscale\undefined%
  \makeatother%
  \begin{picture}(1,0.9170276)%
    \lineheight{1}%
    \setlength\tabcolsep{0pt}%
    \put(0,0){\includegraphics[width=\unitlength,page=1]{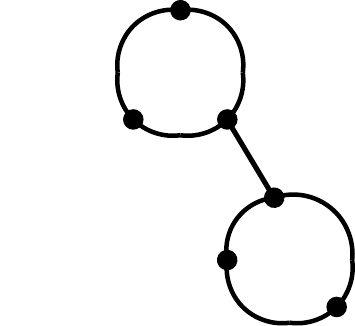}}%
    \put(0.06773293,0.73417891){\color[rgb]{0,0,0}\makebox(0,0)[lt]{\lineheight{1.25}\smash{\begin{tabular}[t]{l}$\Lambda_1$\end{tabular}}}}%
  \end{picture}%
\endgroup%
 \hspace{0.5in}
     \def\svgscale{0.45}
\begingroup%
  \makeatletter%
  \providecommand\color[2][]{%
    \errmessage{(Inkscape) Color is used for the text in Inkscape, but the package 'color.sty' is not loaded}%
    \renewcommand\color[2][]{}%
  }%
  \providecommand\transparent[1]{%
    \errmessage{(Inkscape) Transparency is used (non-zero) for the text in Inkscape, but the package 'transparent.sty' is not loaded}%
    \renewcommand\transparent[1]{}%
  }%
  \providecommand\rotatebox[2]{#2}%
  \newcommand*\fsize{\dimexpr\f@size pt\relax}%
  \newcommand*\lineheight[1]{\fontsize{\fsize}{#1\fsize}\selectfont}%
  \ifx\svgwidth\undefined%
    \setlength{\unitlength}{170.38004237bp}%
    \ifx\svgscale\undefined%
      \relax%
    \else%
      \setlength{\unitlength}{\unitlength * \real{\svgscale}}%
    \fi%
  \else%
    \setlength{\unitlength}{\svgwidth}%
  \fi%
  \global\let\svgwidth\undefined%
  \global\let\svgscale\undefined%
  \makeatother%
  \begin{picture}(1,0.9170276)%
    \lineheight{1}%
    \setlength\tabcolsep{0pt}%
    \put(0,0){\includegraphics[width=\unitlength,page=1]{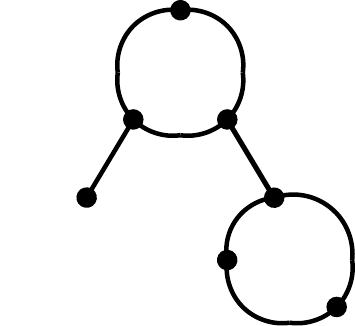}}%
    \put(0.06773293,0.73417891){\color[rgb]{0,0,0}\makebox(0,0)[lt]{\lineheight{1.25}\smash{\begin{tabular}[t]{l}$\Lambda_2$\end{tabular}}}}%
  \end{picture}%
\endgroup%
 \hspace{0.5in}
     \def\svgscale{0.45}
\begingroup%
  \makeatletter%
  \providecommand\color[2][]{%
    \errmessage{(Inkscape) Color is used for the text in Inkscape, but the package 'color.sty' is not loaded}%
    \renewcommand\color[2][]{}%
  }%
  \providecommand\transparent[1]{%
    \errmessage{(Inkscape) Transparency is used (non-zero) for the text in Inkscape, but the package 'transparent.sty' is not loaded}%
    \renewcommand\transparent[1]{}%
  }%
  \providecommand\rotatebox[2]{#2}%
  \newcommand*\fsize{\dimexpr\f@size pt\relax}%
  \newcommand*\lineheight[1]{\fontsize{\fsize}{#1\fsize}\selectfont}%
  \ifx\svgwidth\undefined%
    \setlength{\unitlength}{170.38004237bp}%
    \ifx\svgscale\undefined%
      \relax%
    \else%
      \setlength{\unitlength}{\unitlength * \real{\svgscale}}%
    \fi%
  \else%
    \setlength{\unitlength}{\svgwidth}%
  \fi%
  \global\let\svgwidth\undefined%
  \global\let\svgscale\undefined%
  \makeatother%
  \begin{picture}(1,0.9170276)%
    \lineheight{1}%
    \setlength\tabcolsep{0pt}%
    \put(0,0){\includegraphics[width=\unitlength,page=1]{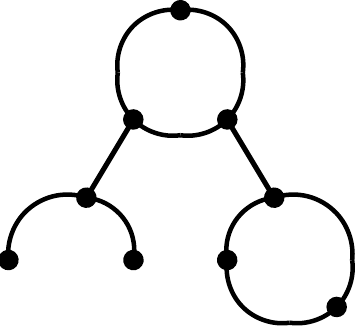}}%
    \put(0.06773293,0.73417891){\color[rgb]{0,0,0}\makebox(0,0)[lt]{\lineheight{1.25}\smash{\begin{tabular}[t]{l}$\Lambda_3$\end{tabular}}}}%
  \end{picture}%
\endgroup%
\\
     \def\svgscale{0.45}
\begingroup%
  \makeatletter%
  \providecommand\color[2][]{%
    \errmessage{(Inkscape) Color is used for the text in Inkscape, but the package 'color.sty' is not loaded}%
    \renewcommand\color[2][]{}%
  }%
  \providecommand\transparent[1]{%
    \errmessage{(Inkscape) Transparency is used (non-zero) for the text in Inkscape, but the package 'transparent.sty' is not loaded}%
    \renewcommand\transparent[1]{}%
  }%
  \providecommand\rotatebox[2]{#2}%
  \newcommand*\fsize{\dimexpr\f@size pt\relax}%
  \newcommand*\lineheight[1]{\fontsize{\fsize}{#1\fsize}\selectfont}%
  \ifx\svgwidth\undefined%
    \setlength{\unitlength}{170.38004237bp}%
    \ifx\svgscale\undefined%
      \relax%
    \else%
      \setlength{\unitlength}{\unitlength * \real{\svgscale}}%
    \fi%
  \else%
    \setlength{\unitlength}{\svgwidth}%
  \fi%
  \global\let\svgwidth\undefined%
  \global\let\svgscale\undefined%
  \makeatother%
  \begin{picture}(1,0.9170276)%
    \lineheight{1}%
    \setlength\tabcolsep{0pt}%
    \put(0,0){\includegraphics[width=\unitlength,page=1]{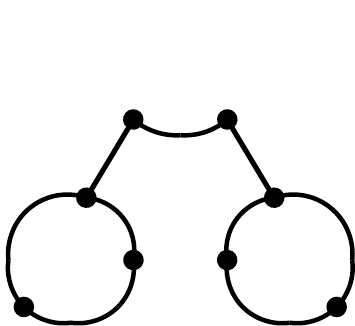}}%
    \put(0.07213501,0.75178662){\color[rgb]{0,0,0}\makebox(0,0)[lt]{\lineheight{1.25}\smash{\begin{tabular}[t]{l}$\Lambda_4$\end{tabular}}}}%
  \end{picture}%
\endgroup%
 \hspace{0.5in}
     \def\svgscale{0.45}
\begingroup%
  \makeatletter%
  \providecommand\color[2][]{%
    \errmessage{(Inkscape) Color is used for the text in Inkscape, but the package 'color.sty' is not loaded}%
    \renewcommand\color[2][]{}%
  }%
  \providecommand\transparent[1]{%
    \errmessage{(Inkscape) Transparency is used (non-zero) for the text in Inkscape, but the package 'transparent.sty' is not loaded}%
    \renewcommand\transparent[1]{}%
  }%
  \providecommand\rotatebox[2]{#2}%
  \newcommand*\fsize{\dimexpr\f@size pt\relax}%
  \newcommand*\lineheight[1]{\fontsize{\fsize}{#1\fsize}\selectfont}%
  \ifx\svgwidth\undefined%
    \setlength{\unitlength}{170.38004237bp}%
    \ifx\svgscale\undefined%
      \relax%
    \else%
      \setlength{\unitlength}{\unitlength * \real{\svgscale}}%
    \fi%
  \else%
    \setlength{\unitlength}{\svgwidth}%
  \fi%
  \global\let\svgwidth\undefined%
  \global\let\svgscale\undefined%
  \makeatother%
  \begin{picture}(1,0.9170276)%
    \lineheight{1}%
    \setlength\tabcolsep{0pt}%
    \put(0,0){\includegraphics[width=\unitlength,page=1]{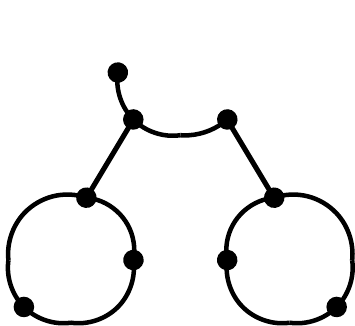}}%
    \put(0.07213501,0.75178662){\color[rgb]{0,0,0}\makebox(0,0)[lt]{\lineheight{1.25}\smash{\begin{tabular}[t]{l}$\Lambda_5$\end{tabular}}}}%
  \end{picture}%
\endgroup%
 \hspace{0.5in}
     \def\svgscale{0.45}
\begingroup%
  \makeatletter%
  \providecommand\color[2][]{%
    \errmessage{(Inkscape) Color is used for the text in Inkscape, but the package 'color.sty' is not loaded}%
    \renewcommand\color[2][]{}%
  }%
  \providecommand\transparent[1]{%
    \errmessage{(Inkscape) Transparency is used (non-zero) for the text in Inkscape, but the package 'transparent.sty' is not loaded}%
    \renewcommand\transparent[1]{}%
  }%
  \providecommand\rotatebox[2]{#2}%
  \newcommand*\fsize{\dimexpr\f@size pt\relax}%
  \newcommand*\lineheight[1]{\fontsize{\fsize}{#1\fsize}\selectfont}%
  \ifx\svgwidth\undefined%
    \setlength{\unitlength}{170.38004237bp}%
    \ifx\svgscale\undefined%
      \relax%
    \else%
      \setlength{\unitlength}{\unitlength * \real{\svgscale}}%
    \fi%
  \else%
    \setlength{\unitlength}{\svgwidth}%
  \fi%
  \global\let\svgwidth\undefined%
  \global\let\svgscale\undefined%
  \makeatother%
  \begin{picture}(1,0.9170276)%
    \lineheight{1}%
    \setlength\tabcolsep{0pt}%
    \put(0,0){\includegraphics[width=\unitlength,page=1]{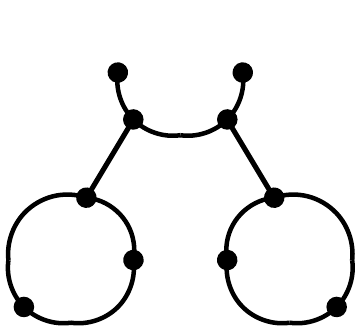}}%
    \put(0.07213501,0.75178662){\color[rgb]{0,0,0}\makebox(0,0)[lt]{\lineheight{1.25}\smash{\begin{tabular}[t]{l}$\Lambda_6$\end{tabular}}}}%
  \end{picture}%
\endgroup%

     \caption{The six candidate subgraphs of $\Delta'$.}
     \label{fig:counterexample_Lambdas}
 \end{figure}
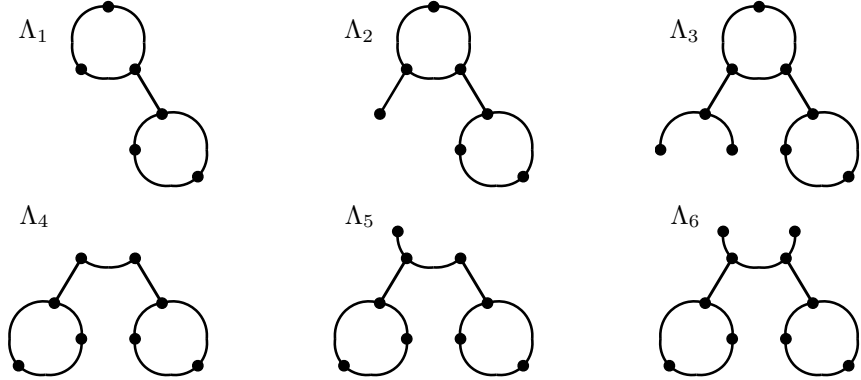

Constructing $UC_2(\Lambda_3)$ and $UC_2(\Lambda_6)$ using Algorithm \ref{alg:comb_config_construction}, we obtain the cube complexes given in Fig. \ref{fig:counterexample_L3_L6}. We therefore see that the corresponding braid groups are
\begin{align*}
B_2(\Lambda_3) &= \pi_1(UC_2(\Lambda_3)) \cong \mathbb{Z}^2 \ast F_5\\
B_2(\Lambda_6) &= \pi_1(UC_2(\Lambda_6)) \cong \mathbb{Z}^2 \ast F_6.
\end{align*}

\begin{figure}[ht]
     \centering
     \def\svgscale{0.45} 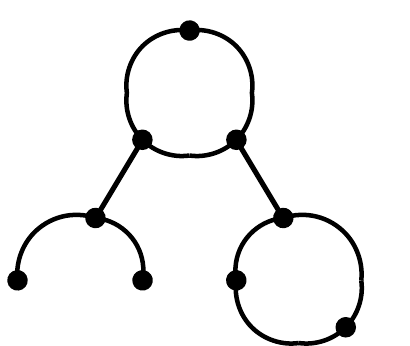 \hspace{1in}
     \def\svgscale{0.675}
     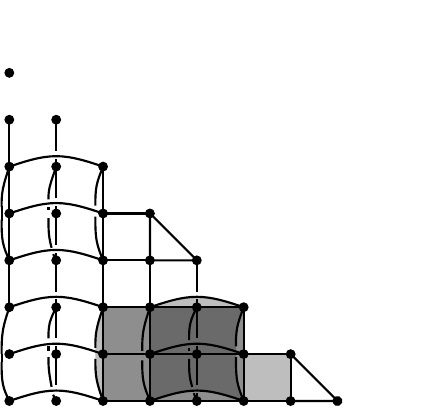 \\
     \vspace{0.2in}
     \def\svgscale{0.45} 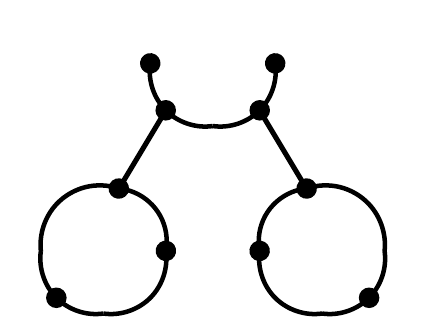 \hspace{0.9in}
     \def\svgscale{0.675}
     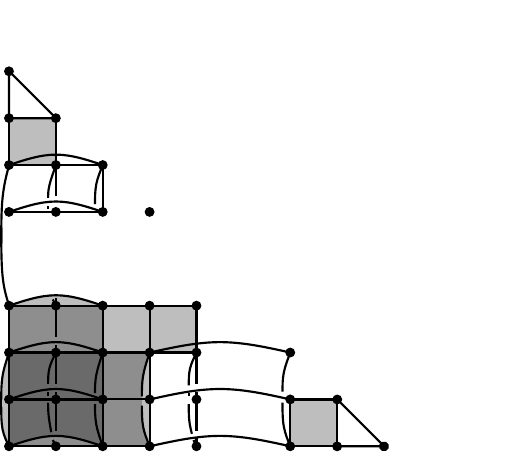 
     \caption{The cube complexes $UC_2(\Lambda_3)$ and $UC_2(\Lambda_6)$.}
     \label{fig:counterexample_L3_L6}
 \end{figure}

Since $P \cong A_\Pi$ does not split as a non-trivial free product and does not embed as a subgroup of $\mathbb{Z}^2$, we conclude that $P$ cannot be contained in $B_2(\Lambda_3)$ or $B_2(\Lambda_6)$ by the Kurosh subgroup theorem \cite[Untergruppensatz]{Kurosh}. Since $\Lambda_1$ and $\Lambda_2$ are proper subgraphs of $\Lambda_3$, and $\Lambda_4$ and $\Lambda_5$ are proper subgraphs of $\Lambda_6$, we see by Lemma \ref{lem:subgraphs_induce_subgroups} that $P$ cannot be contained in $B_2(\Lambda_i)$ for any $1 \leq i \leq 6$. Thus, $P$ cannot be contained in any graph braid group of the form $B_k(\Lambda)$ for $k \leq 2$ and $\Lambda \subsetneq \Delta'$.
\end{proof}

This provides a negative answer to a question of Genevois \cite[follow-up to Question 5.7]{Gen_braids}, showing that while peripheral subgroups arise as braid groups of subgraphs in the case of toral relative hyperbolicity, this is not true in general.

\begin{remark}[Characterising relative hyperbolicity via hierarchical hyperbolicity]\label{rem:isolated_orthog}
Graph braid groups are hierarchically hyperbolic; since they are fundamental groups of special cube complexes, one may apply \cite[Proposition B, Remark 13.2]{BHS_HHSI} to show this. Indeed, a natural explicit hierarchically hyperbolic structure is given in \cite[Section 5.1.2]{Berlyne_Thesis}, described in terms of braid groups on subgraphs. 

Furthermore, Russell shows that hierarchically hyperbolic groups (HHGs) are relatively hyperbolic if and only if they admit an HHG structure $\mf{S}$ with \emph{isolated orthogonality} \cite[Corollary 1.2]{Russell_Rel_Hyp}, a simple criterion restricting where orthogonality may occur in the HHG structure. Roughly, this means there exist $I_1,\dots,I_n \in \mf{S}$ that correspond to the peripheral subgroups.

One may therefore hope that relative hyperbolicity of graph braid groups could be characterized via this approach. However, the above example is problematic here, too; since the peripheral subgroup $P$ does not arise as a braid group on a subgraph of $\Delta'$, it does not appear in the natural HHG structure described in \cite{Berlyne_Thesis}. Thus, Russell's isolated orthogonality criterion is not able to detect relative hyperbolicity of $B_2(\Delta')$ through this HHG structure.
\end{remark}

\subsection{\textit{Distinguishing right-angled Artin groups and graph braid groups}}\label{sec:RAAGs}

Note that Lemmas \ref{lem:gbg_free_product_1} and \ref{lem:gbg_free_product_2} preclude certain right-angled Artin groups and graph braid groups from being isomorphic.

\begin{cor}\label{cor:free_prod_GBGs_are_not_RAAGs}
Let $n \geq 2$, let $\Pi$ be a connected graph and let $\Gamma$ be a finite graph such that one of the following holds:
\begin{itemize}
    \item $\Gamma$ is obtained by gluing a connected non-trivial graph $\Omega$ to a non-segment flower graph $\Phi$ along the central vertex of $\Phi$;
    \item $\Gamma$ satisfies the hypotheses of Theorem \ref{thm:config_space_retraction} and contains an edge $e$ such that $\Gamma \smallsetminus \mathring{e}$ is connected but $\Gamma \smallsetminus e$ is disconnected, and one of the connected components of $\Gamma \smallsetminus e$ is a segment of length at least $n-2$.
\end{itemize}
Then $A_\Pi$ is not isomorphic to $B_n(\Gamma)$.
\end{cor}

\begin{proof}
Since $\Pi$ is connected, the right-angled Artin group $A_\Pi$ does not decompose as a non-trivial free product. However, by Lemmas \ref{lem:gbg_free_product_1} and \ref{lem:gbg_free_product_2}, we know that $B_n(\Gamma)$ must decompose as a non-trivial free product.
\end{proof}

In fact, by combining Lemmas \ref{lem:gbg_free_product_1} and \ref{lem:gbg_free_product_2} with a result of Kim--Ko--Park \cite[Theorem B]{KimKoPark_GBG_Artin}, we obtain stronger results for graph braid groups on large numbers of particles.

\begin{thm}\label{thm:triangle_free_RAAGs_are_not_GBGs}
Let $\Pi$ and $\Gamma$ be finite connected graphs with at least two vertices and let $n \geq 5$. Then the right-angled Artin group $A_\Pi$ is not isomorphic to $B_n(\Gamma)$.
\end{thm}

\begin{proof}
By \cite[Theorem B]{KimKoPark_GBG_Artin}, if $\Gamma$ contains the graph $\Gamma_A$ homeomorphic to the letter ``A'', then $B_n(\Gamma)$ is not isomorphic to a right-angled Artin group. Thus, by sufficiently subdividing edges of $\Gamma$, we may assume that all cycles in $\Gamma$ contain at most one vertex of valence $\geq 3$.

If $\Gamma$ contains a cycle $C$ with precisely one vertex $v$ of valence $\geq 3$, then we can express $\Gamma$ as the graph $C$ glued to the graph $\Omega = \Gamma \smallsetminus (C \smallsetminus v)$ along the vertex $v$. If $\Omega$ is a segment, then $\Gamma$ is homeomorphic to the letter ``Q'', thus Proposition \ref{prop:Q_graph} tells us $B_n(\Gamma) \cong F_n$. Otherwise, we can apply Lemma \ref{lem:gbg_free_product_1} to show that $B_n(\Gamma) \cong H \ast \mathbb{Z}$ for some non-trivial subgroup $H$ of $B_n(\Gamma)$. However, $A_\Pi$ cannot be expressed as a non-trivial free product, since $\Pi$ is connected. Thus, in both cases it follows that $A_\Pi$ is not isomorphic to $B_n(\Gamma)$. We may therefore assume that no cycles in $\Gamma$ contain a vertex of valence $\geq 3$.

If $\Gamma$ contains a cycle $C$ with no vertices of valence $\geq 3$, then $\Gamma$ is either disconnected or $\Gamma = C$. The former case is excluded by hypothesis, while the latter implies $B_n(\Gamma) \cong \mathbb{Z}$. In particular, since $\Pi$ has at least two vertices, this implies $B_n(\Gamma)$ is not isomorphic to $A_\Pi$.

If $\Gamma$ contains no cycles, then $\Gamma$ is a tree. If $\Gamma$ is a segment, then $B_n(\Gamma)$ is trivial. If $\Gamma$ is homeomorphic to a $k$--pronged radial tree $R_k$, then Proposition \ref{prop:radial_trees} tells us $B_n(\Gamma) \cong F_M$ for some $M \geq 10$. Otherwise, $\Gamma$ contains at least two vertices of valence $\geq 3$. Thus, $\Gamma$ can be expressed as $\Gamma = \Phi \cup \Omega$, where $\Phi$ is a segment, $\Omega$ is not a segment, and $\Phi \cap \Omega$ is a vertex of valence $\geq 2$. Thus, we may apply Lemma \ref{lem:gbg_free_product_1} to conclude that $B_n(\Gamma) \cong H \ast \mathbb{Z}$ for some non-trivial subgroup $H$, hence $B_n(\Gamma)$ is not isomorphic to $A_\Pi$.
\end{proof}

\section{Open Questions}\label{sec:questions}

Theorem \ref{thm:hyperbolicity_and_free_products} depends on the assumption that $\Gamma$ is not homeomorphic to a generalized theta graph $\Theta_m$. Furthermore, Proposition \ref{prop:theta_graph} tells us that $B_4(\Gamma_\theta)$ is an HNN extension of $\mathbb{Z} \ast \mathbb{Z}^2$, but it is not clear whether it is a non-trivial free product; knowing this would benefit Theorem \ref{thm:toral_rel_hyp_and_free_prods}. Computing braid groups of generalized theta graphs would therefore be very helpful to improve these two theorems. 

\begin{question}\label{q:gen_theta}
What is $B_n(\Theta_m)$ for $n \geq 3$ and $m \geq 2$? Are they non-trivial free products?
\end{question}

Note, it is easy to show $B_2(\Theta_m)$ is a free group of rank at least $3$ using Theorem \ref{thm:graphs_of_graph_braid_groups}, however the case of $n \geq 3$ is trickier. One potential way of approaching this problem would be to apply Theorem \ref{thm:graphs_of_graph_braid_groups}, letting $v$ be one of the vertices of $\Theta_m$ of valence $m+1$ and decomposing $B_n(\Theta_m)$ as a graph of groups $(\mc{G},\Lambda)$ using $m$ of the edges incident to $v$. Denoting these edges by $e_1, \dots, e_m$, the graphs $\Theta_m \smallsetminus (\mathring{e}_1 \cup \dots \cup \mathring{e}_m)$ and $\Theta_m \smallsetminus ({e}_1 \cup \dots \cup {e}_m)$ are both homeomorphic to the $(m+1)$--pronged radial tree $R_{m+1}$, as illustrated in Fig. \ref{fig:gen_theta}. Thus, by Propositions \ref{prop:counting_edges} and \ref{prop:radial_trees}, $\Lambda$ has a single vertex whose vertex group is a free group of rank $M(n,m+1)$, and $m$ edges whose edge groups are free groups of rank $M(n-1,m+1)$, where $M$ is defined in Proposition \ref{prop:radial_trees}. It remains to understand the monomorphisms mapping the edge groups into the vertex groups.

\begin{figure}[ht]
     \centering
     \def\svgscale{1}
\begingroup%
  \makeatletter%
  \providecommand\color[2][]{%
    \errmessage{(Inkscape) Color is used for the text in Inkscape, but the package 'color.sty' is not loaded}%
    \renewcommand\color[2][]{}%
  }%
  \providecommand\transparent[1]{%
    \errmessage{(Inkscape) Transparency is used (non-zero) for the text in Inkscape, but the package 'transparent.sty' is not loaded}%
    \renewcommand\transparent[1]{}%
  }%
  \providecommand\rotatebox[2]{#2}%
  \newcommand*\fsize{\dimexpr\f@size pt\relax}%
  \newcommand*\lineheight[1]{\fontsize{\fsize}{#1\fsize}\selectfont}%
  \ifx\svgwidth\undefined%
    \setlength{\unitlength}{97.3655977bp}%
    \ifx\svgscale\undefined%
      \relax%
    \else%
      \setlength{\unitlength}{\unitlength * \real{\svgscale}}%
    \fi%
  \else%
    \setlength{\unitlength}{\svgwidth}%
  \fi%
  \global\let\svgwidth\undefined%
  \global\let\svgscale\undefined%
  \makeatother%
  \begin{picture}(1,0.88539606)%
    \lineheight{1}%
    \setlength\tabcolsep{0pt}%
    \put(0,0){\includegraphics[width=\unitlength,page=1]{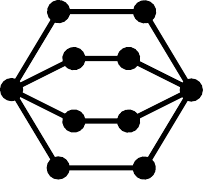}}%
    \put(0.05708158,0.69111676){\color[rgb]{0,0,0}\makebox(0,0)[lt]{\lineheight{1.25}\smash{\begin{tabular}[t]{l}$e_1$\end{tabular}}}}%
    \put(0.18610558,0.59189327){\color[rgb]{0,0,0}\makebox(0,0)[lt]{\lineheight{1.25}\smash{\begin{tabular}[t]{l}$e_2$\end{tabular}}}}%
    \put(0.1995857,0.40803416){\color[rgb]{0,0,0}\makebox(0,0)[lt]{\lineheight{1.25}\smash{\begin{tabular}[t]{l}$e_3$\end{tabular}}}}%
  \end{picture}%
\endgroup%
\hspace{0.5in}\def\svgscale{1}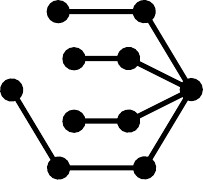
     \caption{The generalized theta graph $\Theta_3$ and the copy of $R_4$ obtained by removing the open edges $\mathring{e}_1,\mathring{e}_2,\mathring{e}_3$.}
     \label{fig:gen_theta}
 \end{figure}
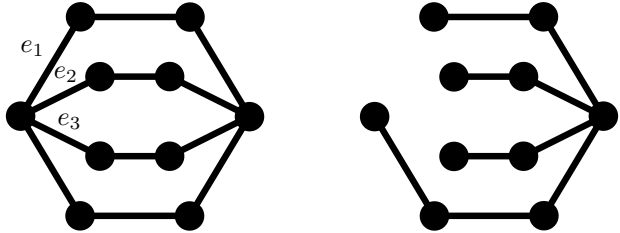

Another interesting question regards the converses of Lemmas \ref{lem:gbg_free_product_1} and \ref{lem:gbg_free_product_2}. The author is not aware of any examples of graph braid groups that decompose as non-trivial free products but do not satisfy either Lemma \ref{lem:gbg_free_product_1} or Lemma \ref{lem:gbg_free_product_2}.

\begin{question}
Are there any graph braid groups that decompose as non-trivial free products but do not satisfy the hypotheses of Lemma \ref{lem:gbg_free_product_1} or Lemma \ref{lem:gbg_free_product_2}?
\end{question}

Since Lemmas \ref{lem:gbg_free_product_1} and \ref{lem:gbg_free_product_2} both produce free splittings with an infinite cyclic factor, an interesting related question concerns the existence of graph braid groups that admit free splittings but do not have a free splitting with an infinite cyclic factor. It seems likely that such graph braid groups should exist, however the answer is not completely obvious. Variants of a theorem of Shenitzer \cite{Shenitzer} show that if a graph of groups splits freely, it can often be arranged that one of the vertex groups splits freely relative to the incoming edge groups; see \cite[Theorem 18]{Wilton_Graphs_of_free_groups} and \cite[Lemma A.6]{CalegariWilton_Random_graphs_of_free_groups}. Moreover, graph braid groups can always be inductively decomposed as graphs of groups until one obtains free groups as the vertex groups. It is therefore not inconceivable that graph braid groups that split freely might always be able to do so with an infinite cyclic factor.

\begin{question}
Are there any graph braid groups $B_n(\Gamma)$ that split as non-trivial free products but do not have a free splitting of the form $B_n(\Gamma) \cong H \ast \mathbb{Z}$?
\end{question}

Corollary \ref{cor:free_prod_GBGs_are_not_RAAGs} and Theorem \ref{thm:triangle_free_RAAGs_are_not_GBGs} also raise a natural question: can right-angled Artin groups with connected defining graphs be isomorphic to graph braid groups $B_n(\Gamma)$ for $n \leq 4$?

\begin{question}
Does there exist some $n \geq 2$ and a finite connected graph $\Gamma$ such that $B_n(\Gamma)$ is isomorphic to a non-cyclic right-angled Artin group with connected defining graph?
\end{question}

Finally, Theorem \ref{thm:counterexample} shows that the peripheral subgroups of a relatively hyperbolic graph braid group $B_n(\Gamma)$ may not be contained in any braid group $B_k(\Lambda)$ for $\Lambda \subsetneq \Gamma$. Thus, while the groups $B_k(\Lambda)$ form a natural collection of subgroups of $B_n(\Gamma)$ by Lemma \ref{lem:subgraphs_induce_subgroups}, they are insufficient to capture the relatively hyperbolic structure. The question of how to classify relative hyperbolicity in graph braid groups therefore remains.

\begin{question}
When is a graph braid group relatively hyperbolic?
\end{question}

One potential approach would be to construct a hierarchically hyperbolic structure on $B_n(\Gamma)$ with greater granularity than the one produced in \cite{Berlyne_Thesis}, and then apply Russell's isolated orthogonality criterion as described in Remark \ref{rem:isolated_orthog}. In particular, one would have to construct a factor system on the cube complex $UC_n(\Gamma)$ that contains enough subcomplexes to guarantee that any peripheral subgroup of $B_n(\Gamma)$ will always appear as the fundamental group of such a subcomplex.

\vspace{3mm}

\noindent \textbf{Acknowledgments.} The author would like to thank Mark Hagen for many enlightening conversations about cube complexes and graphs of groups, Jason Behrstock for valuable discussions on earlier iterations of this research, and Anthony Genevois and Tomasz Maciazek for their helpful comments. The author would also like to thank Jacob Russell, Victor Chepoi, Ben Knudsen, Tim Susse, and Kate Vokes for conversations on graph braid groups that collectively helped to consolidate the author's understanding of the subject.

The author would also like to thank the anonymous referee for their numerous insightful comments that helped to improve this paper.

This work was supported by the Additional Funding Programme for Mathematical Sciences, delivered by EPSRC (EP/V521917/1) and the Heilbronn Institute for Mathematical Research.

\newpage
\let\oldbibliography\thebibliography
\renewcommand{\thebibliography}[1]{\oldbibliography{#1}
\setlength{\itemsep}{0pt}}
{\small\bibliography{Berlyne}{}}
\bibliographystyle{ws-ijac}
\end{document}